\documentclass[a4paper,12pt,draft]{amsart}
\usepackage[margin=30mm]{geometry}
\usepackage{amssymb, amsmath, amsthm, amscd}

\usepackage[T1]{fontenc}
\usepackage{eucal,mathrsfs,dsfont}
\usepackage{color}

\definecolor{mno}{rgb}{0.5,0.1,0.5}

\def\<{\langle}
\def\>{\rangle}
\newcommand{\R}{\mathds R}

\newcommand{\Pp}{\mathds P}
\newcommand{\Ee}{\mathds E}
\newcommand{\I}{\mathds 1}

\newtheorem{theorem}{Theorem}[section]
\newtheorem{lemma}[theorem]{Lemma}
\newtheorem{proposition}[theorem]{Proposition}
\newtheorem{corollary}[theorem]{Corollary}

\theoremstyle{definition}

\newtheorem{example}[theorem]{Example}
\newtheorem{remark}[theorem]{Remark}

\begin{document}

\allowdisplaybreaks

\title[H\"{o}lder Continuous Semigroups
for Operators of Variable Order] {Uniform H\"{o}lder Estimates on
Semigroups Generated by Non-Local Operators of Variable Order}

\author{Dejun Luo\qquad Jian Wang}
\thanks{\emph{D.\ Luo:}
Institute of Applied Mathematics, Academy of Mathematics and Systems Science,
Chinese Academy of Sciences, Beijing 100190, P.R. China.
\texttt{luodj@amss.ac.cn}}
\thanks{\emph{J.\ Wang:}
School of Mathematics and Computer Science, Fujian Normal
University, Fuzhou 350007, P.R. China.
\texttt{jianwang@fjnu.edu.cn}}
\date{}

\begin{abstract} We consider the non-local operator of variable order as follows
  $$Lf(x)= \int_{\R^d\setminus\{0\}}\big(f(x+z)-f(x)-\<\nabla f(x),z\>
  \I_{\{|z|\le 1\}}\big)\frac{n(x,z)}{|z|^{d+\alpha(x)}}\,dz.$$
Under mild conditions on $\alpha(x)$ and $n(x,z)$, we establish the H\"{o}lder
regularity for the associated semigroups. The proof is based on the probabilistic
coupling method, and it successfully applies to both stable-like processes
in the sense of Bass and time-change of symmetric stable processes.
\medskip


\medskip

\end{abstract}

\maketitle

\section{Introduction and Main Result}\label{section1}
In the last few years, more and more people tend to use non-local integral operators
(or, equivalently, processes with jumps) to model problems in mathematical physics and finance,
since in many applications jump processes are more realistic models than continuous processes.
While there is a rich literature concerned with the regularity for diffusion
semigroups generated by second order elliptic differential operators (see e.g. \cite{Gt, Kr, Ce}
and the references therein), the study of regularity properties of the semigroups associated
to non-local operators is far from complete.

In this paper we aim to establish the H\"{o}lder regularity for Markov semigroups associated
with the infinitesimal generator
  \begin{equation}\label{ope-1}
  Lf(x)= \int_{\R^d\setminus\{0\}}\big(f(x+z)-f(x)-\<\nabla f(x), z\>
  \I_{\{|z|\le 1\}}\big)\frac{n(x,z)}{|z|^{d+\alpha(x)}}\,dz.
  \end{equation}
This is a reasonably general integro-differential operator which
includes, for example, many of the operators considered by
probabilists. In probabilistic term, when both $\alpha(x)$ and
$n(x,z)$ are constant functions, the process associated with $L$
given by \eqref{ope-1} is essentially a (rotationally) symmetric
stable process. In our setting, however, the operator $L$ given by
\eqref{ope-1} can be of variable order, i.e. $\alpha(x)$ is a
function of $x$. Therefore, in general the corresponding process
should behave like a symmetric stable process at each point $x$, but
vary from point to point. Though it is known that the semigroups of
symmetric stable processes enjoy nice smooth properties, 
there are many essential differences
between symmetric stable processes and those processes associated with $L$
above, even in the case that $\alpha(x)$ is a constant function, see for instance
\cite{B2, BT, BR}. The purpose of our paper is to
investigate what conditions on $\alpha(x)$ and $n(x,z)$ are needed to
guarantee the H\"{o}lder continuity of the corresponding semigroups.
The functions $\alpha(x)$ and $n(x,z)$ are supposed to be
at least uniformly continuous, see Assumption {\bf(H)} below for
the precise conditions. We stress that, we are
able to treat the case where the index function $\alpha(x)$ is
logarithmical uniformly continuous and the coefficient $n(x,z)$
enjoys the symmetry property, i.e. $n(x,z)=n(x,-z)$ for each $x$,
$z\in\R^d$, or the case that the function $\alpha(x)$ is
logarithmical uniformly continuous such that
$\inf_{x\in\R^d}\alpha(x)>1$. 
In particular, we will
establish H\"{o}lder regularity for Markov semigroups of
stable-like processes in the sense of Bass \cite{B1}.

To state the main result, we first introduce the following assumptions on the operator $L$.

\paragraph{\textbf{Assumption (H)}}
\begin{itemize}
\item[(H1)] The martingale problem associated with the operator $L$ has a unique solution,
and the associated unique $L$-process enjoys the strong Markov property and does not explode
starting from any $x\in\R^d.$
\item[(H2)] There exist two constants $0<\alpha_0\le \alpha_2<2$
such that for all $x\in\R^d$, $\alpha(x)\in [\alpha_0,\alpha_2]$.
The function $\alpha(x)$ satisfies that
  \begin{equation}\label{h2}
  \lim_{|x-y|\to0} |\alpha(x)-\alpha(y)|\log \frac{1}{|x-y|}=0.
  \end{equation}
\item[(H3)] There are two positive constants $c_1$ and $c_2$ such that for all $x$, $z\in\R^d$,
  $$c_1\le n(x,z)\le c_2.$$
The function $(x,z)\mapsto n(x,z)$ is uniformly continuous on $\R^{2d}$ in the sense that
  \begin{equation}\label{h3}
  \lim_{r\to0} \bigg[\sup_{x\in\R^d, |z_1-z_2|\le r} |n(x,z_1)-n(x,z_2)|
  +\sup_{|z|\le 1, |x-y|\le r} |n(x,z)-n(y,z)|\bigg]=0.
  \end{equation}
\item[(H4)] The functions $\alpha(x)$ and $n(x,z)$ satisfy that
  \begin{equation}\label{h4}
  \lim_{r\to 0} \sup_{|x-y|\le r}\!{|x-y|^{\alpha(x)\wedge \alpha(y)-2}}\!
  \int_{\!\left\{\!\frac{|x-y|}{2}<|z|\le1\!\right\}} \!\langle x-y,z\rangle\!\bigg[\frac{n(y,z)}{|z|^{d+\alpha(y)}}-\frac{n(x,z)}{|z|^{d+\alpha(x)}}\bigg]dz=0.
  \end{equation}
\end{itemize}

Let us make some comments on Assumption \textbf{(H)}.

\begin{remark}
(1) There are a few papers that handle with martingale problems for
non-local operators of variable order, e.g.\ see \cite{B1, MP, Ts,
H, K, T1} and the references therein. Among the first is \cite{B1}
by Bass, which proved in the one-dimensional situation
well-posedness of the martingale problem for
$-(-\Delta)^{\alpha(x)}$ under weak assumptions on $\alpha(x)$, see
Example \ref{exm1} below for more details. Following \cite{B1}, we
call the associated pure jump type Markov process the stable-like
process with exponent $\alpha(x)$. In recent years there has been
considerable interest in operators whose jump kernel is of the form
$\frac{n(x,z)}{|z|^{d+\alpha(x)}}\,dz$, e.g.\ see \cite{BL, BK1,
BK2,  BBCK}. The corresponding martingale problem was considered in
\cite{T1}.

(2) The assumption (H2) on the index function $\alpha(x)$ is
standard to ensure the uniqueness of the solution to the martingale
problem for $L$, see \cite[Corollary 2.3]{B1} and \cite[Assumption
2.2]{T1}. In particular, \eqref{h2} holds when $\alpha(x)$ is H\"{o}lder continuous.

(3) According to \cite[Assumption 1.1]{BT} or \cite[Assumption 2.1]{T1}, the condition that the
function $n$ is bounded above and below is natural for the existence of solutions
to the martingale problem of $L$. This is the analogue of uniform
ellipticity and boundedness conditions in the theory of second order
elliptic differential operators. 

(4) The assumption (H4) is our technique condition. It is clear that
\eqref{h4} holds if $n(x,z)=n(x,-z)$ for all $x$, $z\in\R^d$.  Note
that the symmetry of the function $n(x,z)$ with respect to $z$ is a
commonly assumed condition for non-local operators, e.g.\ see
\cite{BL,CS, CZ}. We will see from Proposition \ref{pro-ex} below
that \eqref{h4} is also satisfied in many situations, in particular,
when $\inf_{x\in\R^d}\alpha(x)>1.$
\end{remark}

Let $((\Pp^x)_{x\in\R^d}, (X_t)_{t\ge0})$ be a strong Markov process
such that for each $x\in\R^d$, the probability measure $\Pp^x$ is the
unique solution to the martingale problem for $L$ starting at $x$. Let $\Ee^x$
be the expectation of the process $(X_t)_{t\ge0}$ starting from $x\in\R^d$.
For any $x\in\R^d$ and $t>0$, let
  $$P_tf(x):=\Ee^x(f(X_t)),\quad f\in B_b(\R^d).$$
Now, it is time to present the main
contribution of our paper.

\begin{theorem}\label{th1.1} Under assumption {\bf(H)}, for any
$\beta\in(0,\alpha_0\wedge1)$ there exists a constant $C:=C(\beta,
\alpha(x), n(x,z))>0$ such that for all $f\in B_b(\R^d)$ and $t>0$,
  $$\sup_{x\neq y}\frac{|P_tf(y)-P_tf(x)|}{|y-x|^\beta}\le \frac{C \|f\|_\infty}{(t\wedge 1)^{\beta/\alpha_0}}.$$
In particular, for any $f\in B_b(\R^d)$ and $t>0$, $P_tf$ is
$\beta$-H\"{o}lder continuous with any
$\beta\in(0,\alpha_0\wedge1)$.
\end{theorem}

H\"{o}lder continuity and Harnack inequality of bounded functions
that are harmonic in a domain with respect to pure jump non-local
operators are established in \cite{BK2, BK1}. In fact, the setting
of \cite{BK2} is more general in the sense that the jump measure of
the corresponding operator is not required to have a density with
respect to the Lebesgue measure. In particular, according to
\cite[Theorem 2.2 and Example 2]{BK2}, under assumptions (H1)--(H3)
and the additional condition that either $\inf_{x\in\R^d}
\alpha(x)>1$ or $n(x,z)=n(x,-z)$ for all $x$, $z\in\R^d$ (which
implies that (H4) holds true), we know that any bounded function
which is harmonic in a domain with respect to the operator $L$ given
by \eqref{ope-1} is H\"{o}lder continuous. This, along with
\cite[Proposition 3.1]{BK1} and \cite[Section 4]{BL}, implies that
the associated resolvents are H\"{o}lder continuous too, see e.g.
\cite[Proposition 3.3]{BKK}. If moreover the associated Markov
process is symmetric and possesses bounded transition density
function, then, according to \cite[Proposition 3.4]{BKK} and by
using the spectral theory, we can get the H\"{o}lder continuity of
the corresponding semigroups. However, we can see from \cite[Section
5]{FT} or \cite[Section 4]{SW2014} that the process generated by
the operator $L$ is in general non-symmetric. On the other hand,
consider, for example, the shift semigroup $P_tf(x):= f(x+t)$ which is
certainly not strong Feller, but has a strong Feller resolvent, that is,
the resolvent maps $B_b(\R)$ into $C_b(\R)$.
Indeed, for any $\lambda>0$ and $f\in B_b(\R)$,
  $$U_\lambda f(x)=\int_0^\infty e^{-\lambda t}f(x+t)\, dt
  =e^{\lambda x}\int_x^\infty e^{-\lambda t}f(t)\, dt$$
is globally Lipschitz continuous on $\R$ with Lipschitz
constant $(\lambda+1)\|f\|_\infty$.
This counterexample indicates that, in the non-symmetric situation,
the H\"{o}lder continuity of resolvent operators does not imply that
of the corresponding semigroups. Therefore, Theorem \ref{th1.1} is not
a direct or a simple consequence of the existing results on the
H\"{o}lder continuity of bounded functions that are harmonic in a domain.

As applications of Theorem \ref{th1.1}, we consider the following
two examples.

\begin{example} \label{exm1} Consider the following integro-differential operator
  $$Lf(x)= w(x)\int_{\R^d\setminus\{0\}}\Big(f(x+z)-f(x)-\<\nabla f(x), z\>
  \I_{\{|z|\le 1\}}\Big)|z|^{-d-\alpha(x)}\,dz$$
for $f\in C_c^\infty(\R^d)$. The weight function $w(x)$ is chosen in
such a way that
  $$w(x)= \alpha(x)2^{\alpha(x)-1}\, \frac{\Gamma\big((\alpha(x)+ d)/2\big)}
  {\pi^{d/2}\,\Gamma\big(1-\alpha(x)/2\big)},$$
then $L e_\xi(x)=-|\xi|^{\alpha(x)}e_\xi(x)$, where
$e_\xi(x)=e^{i \langle x,\xi\rangle}$, see e.g.\ \cite[Exercise 18.23, page
184]{ber-for}. With this norming, $L$ can be written as a
pseudo-differential operator $-p(x,D)$ with the symbol
$-|\xi|^{\alpha(x)}$,
  $$Lf(x) = \int e^{i\langle x,\xi\rangle} |\xi|^{\alpha(x)}\widehat f(\xi)\,d\xi
  = -(-\Delta)^{\alpha(x)} f(x),$$
which implies that $L=-(-\Delta)^{\alpha(x)}$ is a stable-like
operator in the sense of Bass \cite{B1}.

For any $r>0$, let
  $$\rho(r):=\sup_{|x-y|\le r} |\alpha(x)-\alpha(y)|$$
be the modulus of continuity of the index function $\alpha$.
Suppose the following conditions hold:
\begin{itemize}
\item[(i)] $0<\alpha_0:=\inf_{x\in\R^d} \alpha(x)\le \sup_{x\in\R^d}\alpha(x)=:\alpha_2<2$;
\item[(ii)] $\lim\limits_{r\to0}\rho(r)|\log r|=0$, and $\int_0^1\rho(r)/r\,dr<\infty.$
\end{itemize} Then, according \cite[Corollary 2.3 and Remark 7.1]{B1},
the martingale problem associated with the operator $L$ is well posed.
It is clear that assumptions (H2) and (H4) are satisfied. In particular, in this
case $n(x,z)=w(x)$ is independent of $z\in\R^d$. On the other hand, since the function
  $$r\mapsto r2^{r-1}\, \frac{\Gamma\big((r+d)/2\big)}{\pi^{d/2}\,
  \Gamma(1-r/2)}$$
is smooth on $[\alpha_0,\alpha_2]\subset(0,2)$, we know from the definition of $w(x)$
and the assumption $\lim_{r\to0}\rho(r)|\log r|=0$ that the function
$w(x)$ is bounded from above and below and uniformly continuous
on $\R^d$. Thus, Assumption (H3) also holds. Therefore, by Theorem
\ref{th1.1}, the associated semigroup of the stable-like process is
$\beta$-H\"{o}lder continuous for any $\beta\in(0,\alpha_0\wedge1).$
\end{example}

\begin{example}\label{exm2} Suppose that $A(x)= (a_{i,j}(x))_{1\le i,j\le d}$
is a bounded continuous $(d\times d)$-matrix-valued function on $\R^d$
that is nondegenrate at every $x\in\R^d$, and $(Z_t)_{t\ge0}$ is a
(rotationally) symmetric $\alpha$-stable process on $\R^d$ for some
$0<\alpha<2$. It is shown in \cite[Theorem 7.1]{BC} that for every
$x\in\R^d$ the stochastic differential equation (SDE)
  \begin{equation}
  \label{exm2-1}dX_t=A(X_{t-})\,dZ_t, \quad X_0=x
  \end{equation}
has a unique weak solution. Although it is assumed in \cite{BC} that
$d\ge 2$, the results here are valid for $d=1$ as well. In
particular, for $d=1$, $(X_t)_{t\ge0}$ is a time-change of symmetric
$\alpha$-stable process. 
Using the It\^{o} formula, one deduces (see the formula above \cite[(7.2)]{BC}) that
$(X_t)_{t\ge0}$ has generator
  $$Lf(x)=\int \Big(f(x+A(x)u)-f(x)-\<\nabla f(x), A(x)u\>\I_{\{|u|\le 1\}}\Big)
  \frac{c_{d,\alpha}}{|u|^{d+\alpha}}\,du,$$
where $c_{d,\alpha}$ is a positive constant depending on $d$ and $\alpha$.
A change of variable formula $z=A(x)u$ yields that
  $$Lf(x)=\int\Big(f(x+z)-f(x)-\langle\nabla f(x), z\rangle\I_{\{|z|\le 1\}}\Big) \frac{k(x,z)}{|z|^{d+\alpha}}\,dz,$$
where
  $$k(x,z)=\frac{c_{d,\alpha}}{|\textrm{det} A(x)|}\left(\frac{|z|}{|A(x)^{-1}z|}\right)^{d+\alpha}.$$
Here, $\textrm{det} A(x)$ is the determinant of the matrix $A(x)$ and $A(x)^{-1}$ is the
inverse of $A(x)$. In particular, according to the argument in \cite[Section 7]{BC}, we
know that the martingale problem for the operator $L$ is well posed. It is also easy to see
that assumptions (H2) and (H4) hold true.

Next, suppose furthermore that $A(x)=(a_{i,j}(x))_{1\le i,j\le d}$
is uniformly continuous, bounded and elliptic (that is, there are
positive constants $\lambda_1$ and $\lambda_2$ such that
$\lambda_1\textrm{Id}\le A(x)\le \lambda_2\textrm{Id}$ for
every $x\in\R^d$). Then, the assumption (H3) is also satisfied by the
definition of $k(x,z)$. Therefore, according to Theorem
\ref{th1.1}, we conclude that the semigroup corresponding to the SDE
\eqref{exm2-1} is $\beta$-H\"{o}lder continuous for any
$\beta\in(0,\alpha\wedge1).$  \end{example}

When $A(x)=(a_{i,j}(x))_{1\le i,j\le d}$ is H\"{o}lder continuous,
Example \ref{exm2} has been studied in \cite[Corollary 1.3]{CZ}.
Indeed, in this case sharp two-sided estimates on the transition
density of the SDE \eqref{exm2-1} are presented there, while
H\"{o}lder estimates (for all $\alpha\in(0,2)$) and  gradient
estimates (for $\alpha\in [1,2))$ on the heat kernel are obtained,
see \cite[(1.9) and (1.15)]{CZ}. In particular, when
$\alpha\in[1,2)$, the associated semigroups are even Lipschitz
continuous. We also refer the reader to \cite{WXZ} for the recent
study of gradient estimates for SDEs driven by multiplicative
L\'{e}vy noise. Though the assertion of Example \ref{exm2} is much
weaker than \cite[Corollary 1.3]{CZ}, we want to stress that our
approach is completely different from that of \cite{CZ}, and we
indeed only require the uniform continuity of the coefficient
$A(x)=(a_{i,j}(x))_{1\le i,j\le d}$. This fact convinces us that the
coupling method could yield regularity of semigroups associated with
non-local operators with coefficients of low regularity.
Indeed,
when $\alpha_0:=\inf_{x\in\R^d}\alpha(x)>1$, and the functions
$\alpha(x)$ and $n(x,z)$ fulfill stronger continuity, the coupling
method allows us to establish better regularity properties
(including the Lipschitz continuity) for the Markov semigroups
associated with the operator $L$ given by \eqref{ope-1}. For this,
we need the following two classes of reference functions:
  \begin{align*}
  \mathscr{D}=&\left\{\varphi\in C^2((0,2])\cap L^1((0,2];dx): \varphi>0, \varphi'<0\textrm{ and } \varphi''>0 \right\}\\
    \mathscr{D}_\theta=&\left\{\varphi\in \mathscr{D}: \lim_{r\to0} \left[\frac{r^{\theta-2}}{\varphi'(r)}+\frac{\varphi'(2r)r}{\varphi(r)}\right]\!=0 \textrm{ and }\! \limsup_{r\to0} \frac{\varphi''(r)r}{\varphi'(r)} <\theta-2\right\}\!, \theta\in(0,2).
  \end{align*}

\begin{theorem}\label{f-thm} Assume that $\alpha_0:=\inf_{x\in\R^d}\alpha(x)>1$, and assumptions
${\rm(H1)}$--${\rm(H3)}$ hold. Let $\varphi\in \mathscr{D}_{\alpha_0}$ and
  $$\Psi(r):=-\frac{\varphi'(2r)r}{\varphi(r)},\quad r\in(0,1].$$ For any $x,y\in\R^d$, set
\begin{align*}
 A(x,y):=&|\alpha(x)-\alpha(y)|\left(\log \frac{1}{|x-y|}\right)+\sup_{|z|\le 1}
  |n(x,z)-n(y,z)|\\
  &+\sup_{z\in\R^d, |z_1-z_2|\le |x-y|}|n(z,z_1)-n(z,z_2)|.
 \end{align*}
If
  \begin{equation}\label{f-cor-2}\begin{split}
  \lim_{|x-y|\to0} \frac {A(x,y)}{\Psi(|x-y|)}=0,
  \end{split}
  \end{equation}
then there exist constants $\varepsilon_0\in(0,1)$ and $C>0$ such that for all $f\in B_b(\R^d)$, $t>0$ and $|x-y|\leq 1$,
 $$\sup_{x\neq y}\frac{{|P_t f(x)-P_t f(y)|}}{\int_0^{|x-y|}\varphi(s)\,ds}\le
  C\|f\|_\infty \inf_{\varepsilon\in(0,\varepsilon_0]} \left[\frac{1}{\int_0^\varepsilon\varphi(s)\,ds}-\frac{1}{t \varphi'(2\varepsilon) \varepsilon^{2-\alpha_0}}\right].$$
\end{theorem}

As applications of Theorem \ref{f-thm}, we have the following two
typical examples. The first one deals with the Lipschitz continuity,
and it follows from Theorem \ref{f-thm} by taking
$\varphi(r)=1-1/(\log\log(54/r))$; while the second one treats the
log-Lipschitz continuity, and it is a consequence of Theorem
\ref{f-thm} by taking $\varphi(r)=\log^\beta(6/r)$ for any
$\beta>0$.

\begin{example}\label{ttexp}
\begin{itemize}
\item[(1)] If
   \begin{equation}\label{rrtt}
  \begin{split}
 \lim_{|x-y|\to0}\left[ \log \frac{1}{|x-y|} \left(\log\log\frac{1}{|x-y|}\right)^2\right]
 A(x,y)=0,
  \end{split}\end{equation} then there exists a constant
$C>0$ such that for all
$f\in B_b(\R^d)$ and $t>0$,
  $$\sup_{x\neq y}\frac{|P_tf(y)-P_tf(x)|}{|y-x|}\le \frac{C \|f\|_\infty }{(t\wedge e^{-2})^{1/\alpha_0}|\log(t\wedge e^{-2})|^{-1/\alpha_0}\big|\log|\log(t\wedge e^{-2})|\big|^{-2/\alpha_0}}.$$
In particular, for any $f\in B_b(\R^d)$ and $t>0$, $P_tf$ is
Lipschitz continuous.

\item[(2)] If   \begin{align*}
 \lim_{|x-y|\to0}\left( \log \frac{1}{|x-y|}\right) A(x,y)=0,
 \end{align*}
 then, for any $\beta>0$, there exists a constant
$C>0$ such that for all $f\in
B_b(\R^d)$ and $t>0$,
  $$\sup_{x\neq y}\frac{|P_tf(y)-P_tf(x)|}{|y-x|\cdot \big|\log|x-y|\big|^\beta}\le \frac{C
   \|f\|_\infty}{(t\wedge e^{-1})^{1/\alpha_0}{|\log(t\wedge e^{-1})|}^{-1/\alpha_0+\beta}}.$$
\end{itemize}
\end{example}

\begin{remark} Comparing Example \ref{ttexp}(1) with Theorem \ref{th1.1} and Example \ref{ttexp}(2),
we can find that in order to yield the Lipschitz continuity of the Markov semigroups, weak
continuity assumptions on the functions $\alpha(x)$ and $n(x,z)$ are required. On the one hand,
by Theorem \ref{f-thm}, the associated semigroup is still Lipschitz continuous if
\eqref{rrtt} is replaced by the following weaker condition
  $$\lim_{|x-y|\to0}\left[ \log \frac{1}{|x-y|}\left(\log\log\frac{1}{|x-y|} \right)\left(\log\log\log\frac{1}{|x-y|}\right)^2\right] A(x,y)=0,$$
and so on. On the other hand, if there exists some constant $\delta\in(0,1]$ such that
  \begin{align*}\lim_{|x-y|\to0} \frac{1}{|x-y|^\delta}
  \bigg[&|\alpha(x)-\alpha(y)|+\sup_{|z|\le 1}
  |n(x,z)-n(y,z)|\\
  &+\sup_{z\in\R^d, |z_1-z_2|\le |x-y|}|n(z,z_1)-n(z,z_2)|\bigg]=0,
  \end{align*}
i.e. the functions $\alpha(x)$ and $n(x,z)$ are $\delta$-H\"{o}lder
continuous, then \eqref{rrtt} holds too.
\end{remark}

Quite a lot progress has been made during the last decades on heat
kernel estimates for non-local operators, see e.g. \cite{CK2003, BJ,
CK2008, CKK1, CZ} and the references therein. According to the
remark above and motivated by \cite{CZ}, one may wish to construct
the fundamental solution for the operator $L$ defined by
\eqref{ope-1}, and to establish sharp two-sided estimates as well as
its fractional derivative and gradient estimates, which immediately
lead to some regularity properties of the associated semigroups.
However, as mentioned above, since the operator $L$ has variable
order and is generally non-symmetric, it remains a big challenge to
search for heat kernel estimates on it. To the authors' knowledge,
there is very few result on this topic, see \cite{Kolo}. This once
again illustrates the power of coupling
method for non-local operators with variable order. 
To show clearly the practicality of Theorem \ref{th1.1},
we present the following sufficient conditions for assumption (H4).

\begin{proposition}\label{pro-ex} Assume that ${\rm (H2)}$ and ${\rm (H3)}$ hold.
Then, we have the following two statements.
\begin{itemize}
\item[(1)] If $\alpha_0:=\inf_{x\in\R^d}\alpha(x)>1$, then \eqref{h4} is fulfilled.
\item[(2)] If $\alpha_2:=\sup_{x\in\R^d}\alpha(x)<1$, and
  \begin{equation}\label{ttt}
  \lim_{r\to0}\sup_{|x-y|\le r} \frac{|\alpha(x)-\alpha(y)|+\sup_{|z|\le 1}|n(x,z)-n(y,z)|}
  {|x-y|^{1-\alpha(x)\wedge\alpha(y)}}=0,\end{equation}
then \eqref{h4} is satisfied. In particular, if the two functions $\alpha(x)$ and
$x\mapsto n(x,z)$ are Lipschitz continuous (uniformly in $z\in B(1):=\{y\in\R^d:|y|\leq 1\}$),
then \eqref{ttt} holds true.
\end{itemize}
\end{proposition}

To prove Theorem \ref{th1.1}, the main technique we adopt here is
the coupling method as in \cite{CL, Cr, PW} for diffusion processes.
Recently, there are some progress on the coupling property of
L\'{e}vy processes, see e.g.\ \cite{W1,BSW,SSW}. All these papers
rely heavily on the fact that the associated operators are local
operators or the particular characterization of L\'{e}vy processes.
Since the corresponding integro-differential operator $L$ given by
\eqref{ope-1} is of variable order, our techniques differ
significantly from the papers cited above. In particular, our new
coupling is constructed in the following way: by reflection for
small jumps and by parallel displacement (called march coupling in
\cite{CL}) for large jumps, see the next section for more details.
Let us mention that, so far, our paper seems to be the first attempt
to study regularity properties of the semigroups of purely nonlocal L\'{e}vy
type operators by using the coupling method.

The remainder of this paper is arranged as follows. In the next section, we shall
construct a new Markov coupling operator $\tilde L$ of $L$, and prove that there exists a
coupling process associated with it. The last section is devoted to the proofs of Theorem
\ref{th1.1} and other assertions in Section \ref{section1}.

\section{Coupling Operator and Coupling Process}\label{section22}

We begin with the construction of a Markov coupling operator for the
generator $L$ given by \eqref{ope-1}. First, for any $x$, $y$ and
$z\in\R^d$, set
  $$\varphi_{x,y}(z):=
  \begin{cases}
  z-\frac{2\langle x-y, z\rangle}{|x-y|^2}(x-y), & x\neq y;\\
  -z, &x=y.\\
  \end{cases}
  $$
It is clear that $\varphi_{x,y}:\R^d\to\R^d$ enjoys the following three properties:
\begin{itemize}
\item[(A1)] $\varphi_{x,y}(z)=\varphi_{y,x}(z)$ and $\varphi^2_{x,y}(z)=z$, i.e.\
$\varphi_{x,y}^{-1}(z)=\varphi_{x,y}(z)$;
\item[(A2)] $|\varphi_{x,y}(z)|=|z|$;
\item[(A3)] $(z-\varphi_{x,y}(z))\, /\!/ \, (x-y)$ and $(z+\varphi_{x,y}(z))\perp (x-y).$
\end{itemize}

Next, for any $x$, $y$ and $z\in\R^d$, let
  $$\tilde{n}(x,y,z):=n(x,z)\wedge n(y,z)\wedge n(x,\varphi_{x,y}(z))\wedge n(y,\varphi_{x,y}(z)),$$
and for any $f\in C_b^2(\R^{2d}),$ let
  $$\nabla_xf(x,y):=\left(\frac{\partial f(x,y)}{\partial x_i}\right)_{1\le i\le d},\quad \nabla_yf(x,y):=
  \left(\frac{\partial f(x,y)}{\partial y_i}\right)_{1\le i\le d}.$$

Now, for any $f\in C_b^2(\R^{2d})$, we define
  \begin{align*}
  \widetilde{L}_1f(x,y)&:=\frac{1}{2}\bigg[\int_{\left\{|z|\le \frac{|x-y|}{2}\right\}} \!\!\Big(f(x+z,y+\varphi_{x,y}(z))-f(x,y)-\<\nabla_xf(x,y), z\> \I_{\{|z|\le1\}}\\
  &\hskip83pt -\<\nabla_yf(x,y),\varphi_{x,y}(z)\> \I_{\{|z|\le1\}}\Big)
  \frac{\tilde{n}(x,y,z)}{|z|^{d+\alpha(x)}\vee |z|^{d+\alpha(y)}}\,dz\\
  &\hskip15pt +\int_{\left\{|z|\le \frac{|x-y|}{2}\right\}}\!\! \Big( f(x+\varphi_{x,y}(z),y+z)-f(x,y)
  -\<\nabla_yf(x,y), z\> \I_{\{|z|\le1\}}\\
  &\hskip83pt -\<\nabla_xf(x,y), \varphi_{x,y}(z)\> \I_{\{|z|\le1\}}\Big)
  \frac{\tilde{n}(x,y,z)}{|z|^{d+\alpha(x)}\vee |z|^{d+\alpha(y)}}\,dz\bigg]\\
  &\hskip15pt +\int_{\left\{|z|\le \frac{|x-y|}{2}\right\}}\Big(f(x+z,y)-f(x,y)-\<\nabla_xf(x,y), z\> \I_{\{|z|\le1\}}\Big)\,\\
  &\hskip83pt \times\bigg(\frac{n(x,z)}{|z|^{d+\alpha(x)}}-\frac{\tilde{n}(x,y,z)}{|z|^{d+\alpha(x)}\vee |z|^{d+\alpha(y)}}\bigg)\,dz\\
  &\hskip15pt +\int_{\left\{|z|\le \frac{|x-y|}{2}\right\}}\Big(f(x,y+z)-f(x,y)-\<\nabla_yf(x,y), z\> \I_{\{|z|\le1\}}\Big)\,\\
  &\hskip83pt \times\bigg(\frac{n(y,z)}{|z|^{d+\alpha(y)}}-\frac{\tilde{n}(x,y,z)}{|z|^{d+\alpha(x)}\vee |z|^{d+\alpha(y)}}\bigg)\,dz
  \end{align*}
and
  \begin{align*}
  \widetilde{L}_2f(x,y):=&\int_{\left\{|z|>\frac{|x-y|}{2}\right\}}\Big(f(x+z,y+z)-f(x,y)-\<\nabla_xf(x,y), z\> \I_{\{|z|\le1\}}\\
  &\hskip62pt -\langle\nabla_yf(x,y), z\rangle \I_{\{|z|\le1\}}\Big)\,\frac{n(x,z)\wedge n(y,z)}{|z|^{d+\alpha(x)}\vee |z|^{d+\alpha(y)}}\,dz\\
  &+\int_{\left\{|z|>\frac{|x-y|}{2}\right\}}\Big(f(x+z,y)-f(x,y)-\<\nabla_xf(x,y), z\> \I_{\{|z|\le1\}}\Big)\\
  &\hskip65pt \times\bigg(\frac{n(x,z)}{|z|^{d+\alpha(x)}}-\frac{n(x,z)\wedge n(y,z)}{|z|^{d+\alpha(x)}\vee |z|^{d+\alpha(y)}}\bigg)\,dz \\
  &+\int_{\left\{|z|>\frac{|x-y|}{2}\right\}}\Big(f(x,y+z)-f(x,y)-\<\nabla_yf(x,y), z\> \I_{\{|z|\le1\}}\Big)\\
  &\hskip68pt \times\bigg(\frac{n(y,z)}{|z|^{d+\alpha(y)}}-\frac{n(x,z)\wedge n(y,z)}{|z|^{d+\alpha(x)}\vee |z|^{d+\alpha(y)}}\bigg)\,dz.
    \end{align*}

Finally, for any $f\in C_b^2(\R^{2d})$, define
  \begin{equation}\label{coup-1}
  \widetilde{L} f(x,y):=\widetilde{L}_1 f(x,y)+\widetilde{L}_2 f(x,y).
  \end{equation}
We can conclude that

\begin{lemma}
The operator $ \widetilde{L}$ defined by \eqref{coup-1} is the coupling operator
of the operator $L$ given by \eqref{ope-1}, i.e. for any $g,$ $h\in C_b^2(\R^d)$ and $x$,
$y\in\R^d$,
  $$ \widetilde{L} H(x,y)=L g(x)+L h(y),$$
where $H(x,y)=g(x)+h(y).$
\end{lemma}

\begin{proof} Since $\widetilde{L}$ is a linear operator, it suffices to verify that
  $$ \widetilde{L} f(x)=Lf(x),\quad f\in C_b^2(\R^d),$$
where, on the left hand side, $f$ is regarded as a bivariate function on $\R^{2d}$.

First we have
  \begin{align*}
  \widetilde{L}_1 f(x)
  &=\frac{1}{2}\bigg[\int_{\left\{|z|\le \frac{|x-y|}{2}\right\}} \Big(f(x+z)-f(x)-\<\nabla f(x), z\> \I_{\{|z|\le1\}}\Big)\\
  &\hskip87pt \times \frac{\tilde{n}(x,y,z)}{|z|^{d+\alpha(x)}\vee |z|^{d+\alpha(y)}}\,dz\\
  &\hskip24pt +\int_{\left\{|z|\le \frac{|x-y|}{2}\right\}} \Big( f(x+\varphi_{x,y}(z))-f(x)-\<\nabla f(x), \varphi_{x,y}(z)\> \I_{\{|z|\le1\}}\Big)\\
  &\hskip90pt \times\,\frac{\tilde{n}(x,y,z)}{|z|^{d+\alpha(x)}\vee |z|^{d+\alpha(y)}}\,dz\bigg]\\
  &\hskip13pt +\int_{\left\{|z|\le \frac{|x-y|}{2}\right\}}\Big(f(x+z)-f(x)-\<\nabla f(x), z\> \I_{\{|z|\le1\}}\Big)\,\\
  &\hskip82pt \times\bigg(\frac{n(x,z)}{|z|^{d+\alpha(x)}}-\frac{\tilde{n}(x,y,z)}
  {|z|^{d+\alpha(x)}\vee |z|^{d+\alpha(y)}}\bigg)\,dz.
  \end{align*}
By (A1) and the definition of $\tilde{n}(x,y,z)$, we know that
the kernel $\frac{\tilde{n}(x,y,z)}{|z|^{d+\alpha(x)}\vee
|z|^{d+\alpha(y)}}\,dz$ is invariant under the transformation
$z\mapsto \varphi_{x,y}(z)$. This, along with (A2) and the equality
above, leads to
  \begin{align*}
  \widetilde{L}_1 f(x)
  &=\int_{\left\{|z|\le \frac{|x-y|}{2}\right\}} \!\!\Big(f(x+z)-f(x)-\<\nabla f(x), z\> \I_{\{|z|\le1\}}\Big)\\
  &\hskip75pt \times\,\frac{\tilde{n}(x,y,z)}{|z|^{d+\alpha(x)}\vee |z|^{d+\alpha(y)}}\,dz\\
  &\hskip13pt +\int_{\left\{|z|\le \frac{|x-y|}{2}\right\}}\Big(f(x+z)-f(x)-\<\nabla f(x), z\> \I_{\{|z|\le1\}}\Big)\,\\
  &\hskip75pt \times\bigg(\frac{n(x,z)}{|z|^{d+\alpha(x)}}-\frac{\tilde{n}(x,y,z)}
  {|z|^{d+\alpha(x)}\vee |z|^{d+\alpha(y)}}\bigg)\,dz\\
  &=\int_{\left\{|z|\le \frac{|x-y|}{2}\right\}}\Big(f(x+z)-f(x)-\<\nabla f(x), z\> \I_{\{|z|\le1\}}\Big) \frac{n(x,z)}{|z|^{d+\alpha(x)}} \,dz.
  \end{align*}
Next
  \begin{align*}
  \widetilde{L}_2 f(x)&=\int_{\left\{|z|>\frac{|x-y|}{2}\right\}}\Big(f(x+z)-f(x)-\langle\nabla f(x), z\rangle \I_{\{|z|\le1\}}\Big)\,\\
  &\hskip75pt \times\frac{n(x,z)\wedge n(y,z)}{|z|^{d+\alpha(x)}\vee |z|^{d+\alpha(y)}}\,dz\\
  &\hskip13pt +\int_{\left\{|z|>\frac{|x-y|}{2}\right\}}\Big(f(x+z)-f(x)-\langle\nabla f(x),z\rangle \I_{\{|z|\le1\}}\Big)\\
  &\hskip75pt \times\bigg(\frac{n(x,z)}{|z|^{d+\alpha(x)}}-\frac{n(x,z)\wedge n(y,z)}{|z|^{d+\alpha(x)}\vee |z|^{d+\alpha(y)}}\bigg)\,dz \\
  &=\int_{\left\{|z|>\frac{|x-y|}{2}\right\}}\Big(f(x+z)-f(x)-\langle\nabla f(x), z \rangle \I_{\{|z|\le1\}}\Big)\frac{n(x,z)}{|z|^{d+\alpha(x)}}\,dz.
  \end{align*}
Combining the above two equalities finishes the proof.
\end{proof}

In the remainder of this section, we will construct a coupling
process associated with the coupling operator $ \widetilde{L}$
defined by \eqref{coup-1}. For any $x$, $y\in\R^d$ and $A\in
\mathscr{B}(\R^{2d})$, set
  \begin{align*}
  {\mu}(x,y,A)&:=\frac{1}{2}\int_{\left\{(z,\varphi_{x,y}(z))\in A,|z|\le\frac{|x-y|}{2}\right\}}\,
  \frac{\tilde{n}(x,y,z)}{|z|^{d+\alpha(x)}\vee |z|^{d+\alpha(y)}}\,dz\\
  &\hskip14pt +\frac{1}{2}\int_{\left\{(\varphi_{x,y}(z),z)\in
  A,|z|\le\frac{|x-y|}{2}\right\}}\,\frac{\tilde{n}(x,y,z)}{|z|^{d+\alpha(x)}\vee |z|^{d+\alpha(y)}}\,dz\\
  &\hskip14pt  +\int_{\left\{(z,0)\in
  A,|z|\le\frac{|x-y|}{2}\right\}}\,\left( \frac{n(x,z)}{|z|^{d+\alpha(x)}}-
  \frac{\tilde{n}(x,y,z)}{|z|^{d+\alpha(x)}\vee |z|^{d+\alpha(y)}}\right)\,dz\\
  &\hskip14pt  +\int_{\left\{(0,z)\in
  A,|z|\le\frac{|x-y|}{2}\right\}}\,\left( \frac{n(y,z)}{|z|^{d+\alpha(y)}}-
  \frac{\tilde{n}(x,y,z)}{|z|^{d+\alpha(x)}\vee |z|^{d+\alpha(y)}}\right)\,dz\\
  &\hskip14pt +\int_{\left\{(z,z)\in A,|z|>\frac{|x-y|}{2}\right\}}\,
  \frac{n(x,z)\wedge n(y,z)}{|z|^{d+\alpha(x)}\vee |z|^{d+\alpha(y)}}\,dz \\
  &\hskip14pt +\int_{\left\{(z,0)\in A,|z|>\frac{|x-y|}{2}\right\}}\,\bigg(\frac{n(x,z)}{|z|^{d+\alpha(x)}}-\frac{n(x,z)\wedge n(y,z)}{|z|^{d+\alpha(x)}\vee |z|^{d+\alpha(y)}}\bigg)\,dz\\
  &\hskip14pt +\int_{\left\{(0,z)\in A,|z|>\frac{|x-y|}{2}\right\}}\,\bigg(\frac{n(y,z)}{|z|^{d+\alpha(y)}}-\frac{n(x,z)\wedge n(y,z)}{|z|^{d+\alpha(x)}\vee |z|^{d+\alpha(y)}}\bigg)\,dz .
  \end{align*}
Then, by \eqref{coup-1}, for any $x$, $y\in\R^d$ and $f\in C^2_b(\R^{2d})$,  we have
  $$\aligned \widetilde{L} f(x,y) =&\int_{\R^{2d}}\Big[f\big((x,y)+(u_1,u_2)\big)-f(x,y)\\
  &\qquad\,\, -\big\langle \big(\nabla_xf(x,y),\nabla_yf(x,y)\big),(u_1,u_2)\big\rangle \I_{\{|u_1|\le 1, |u_2|\le 1\}}\Big]\,{\mu}(x,y,du_1,du_2).\endaligned$$

For any $h\in C_b(\R^{2d})$, by (A2),
  \begin{align*}
  &\int_{\R^{2d}} h(u)\frac{|u|^2}{1+|u|^2}\,{\mu}(x,y,du)\\
  &=\int_{\left\{|z|\le\frac{|x-y|}{2}\right\}}h(z,\varphi_{x,y}(z))\frac{|z|^2}{1+2|z|^2}
  \cdot \frac{\tilde{n}(x,y,z)}{|z|^{d+\alpha(x)}\vee |z|^{d+\alpha(y)}}\,dz\\
  &\hskip13pt +\int_{\left\{|z|\le\frac{|x-y|}{2}\right\}}h(\varphi_{x,y}(z),z)\frac{|z|^2}{1+2|z|^2}
  \cdot \frac{\tilde{n}(x,y,z)}{|z|^{d+\alpha(x)}\vee |z|^{d+\alpha(y)}}\,dz\\
  &\hskip13pt +\int_{\left\{|z|\le\frac{|x-y|}{2}\right\}}h((z,0))\frac{|z|^2}{1+|z|^2}
  \,\left( \frac{n(x,z)}{|z|^{d+\alpha(x)}}-
  \frac{\tilde{n}(x,y,z)}{|z|^{d+\alpha(x)}\vee |z|^{d+\alpha(y)}}\right)\,dz\\
  &\hskip13pt +\int_{\left\{|z|\le\frac{|x-y|}{2}\right\}}h((0,z))\frac{|z|^2}{1+|z|^2}
  \,\left( \frac{n(y,z)}{|z|^{d+\alpha(y)}}-
  \frac{\tilde{n}(x,y,z)}{|z|^{d+\alpha(x)}\vee |z|^{d+\alpha(y)}}\right)\,dz\\
  &\hskip13pt +2\int_{\left\{|z|>\frac{|x-y|}{2}\right\}}h(z,z)\frac{|z|^2}{1+2|z|^2}
  \cdot \frac{n(x,z)\wedge n(y,z)}{|z|^{d+\alpha(x)}\vee |z|^{d+\alpha(y)}}\,dz\\
  &\hskip13pt +\int_{\left\{|z|>\frac{|x-y|}{2}\right\}}h((z,0))\frac{|z|^2}{1+|z|^2}
  \,\bigg(\frac{n(x,z)}{|z|^{d+\alpha(x)}}-\frac{n(x,z)\wedge n(y,z)}{|z|^{d+\alpha(x)}\vee |z|^{d+\alpha(y)}}\bigg)\,dz\\
  &\hskip13pt +\int_{\left\{|z|>\frac{|x-y|}{2}\right\}}h((0,z))\frac{|z|^2}{1+|z|^2}
  \,\bigg(\frac{n(y,z)}{|z|^{d+\alpha(y)}}-\frac{n(x,z)\wedge n(y,z)}{|z|^{d+\alpha(x)}\vee |z|^{d+\alpha(y)}}\bigg)\,dz ,
  \end{align*}
which, along with the facts that $\alpha(x)$ and $n(x,z)$ are
continuous and strictly positive, implies that $(x,y)\mapsto\int
h(u)\frac{|u|^2}{1+|u|^2}\,{\mu}(x,y,du)$ is a continuous function
on $\R^{2d}$.

According to \cite[Theorem 2.2]{St}, there exist a probability space
$(\widetilde{\Omega}, \widetilde{\mathscr{F}},
(\widetilde{\mathscr{F}}_t)_{t\ge0}, \widetilde{\Pp})$ and an
$\bar{\R}^{2d}$-valued process $(\widetilde{X}_t)_{t\ge0}$ such that
$(\widetilde{X}_t)_{t\ge0}$ is
$(\widetilde{\mathscr{F}}_t)_{t\ge0}$-progressively measurable, and
for every $f\in C_b^2(\R^{2d})$,
  $$\bigg\{f(\widetilde{X}_t)-\int_0^{t\wedge e} \widetilde{L}f(\widetilde{X}_u)\,du, t\ge 0\bigg\}$$
is an $(\widetilde{\mathscr{F}}_t)_{t\ge0}$-local martingale, where
$e$ is the explosion time of $(\widetilde{X}_t)_{t\ge0}$, i.e.\
  $$e=\lim_{n\to\infty}\inf\Big\{t\ge0: |X_t'|+|X_t''|\ge n\Big\}.$$
Here, $(\widetilde{X}_t)_{t\ge0}:=(X_t',X_t'')_{t\ge0}$, and
$(X_t')_{t\ge0}$ and $(X_t'')_{t\ge0}$ are two stochastic processes
on $\R^{d}$. Since $\widetilde{L}$ is the coupling operator of $L$,
the generator of both marginal processes $(X_t')_{t\ge0}$ and
$(X_t'')_{t\ge0}$ is just the operator $L$, and hence both are
solutions to the martingale problem of $L$. In particular, by
(H1), the processes $(X_t')_{t\ge0}$ and $(X_t'')_{t\ge0}$ are
non-explosive, hence one has $e=\infty$ a.s.\, Therefore, the
coupling operator $\widetilde{L}$ generates a non-explosive process
$(\widetilde{X}_t)_{t\ge0}$.

Let $T$ be the coupling time of $(X_t')_{t\ge0}$ and $(X_t'')_{t\ge0}$, i.e.\
  $$T=\inf\{t\ge0: X_t'=X_t''\}.$$
Then $T$ is an $(\widetilde{\mathscr{F}}_t)_{t\geq 0}$-stopping time. Define
a new process $(Y'_t)_{t\ge0}$ as follows
  $$Y_t'=
  \begin{cases}
  X_t'', & t< T;\\
  X_t', & t\ge T.\\
  \end{cases}
  $$
Following the argument of \cite[Section 3.1]{PW}, we conclude that
$(X_t',Y_t')_{t\ge0}$ is also a non-explosive coupling process of
$(X_t)_{t\ge0}$ such that $X_t'=Y_t'$ for any $t\ge T$ and
the generator of $(X_t',Y_t')_{t\ge0}$ before the coupling time $T$
is just the coupling operator $\tilde{L}$ given by \eqref{coup-1}. On
the other hand, according to (H1) and \cite[Lemma 2.1]{CL}, we know that
for any $x$, $y\in\R^d$ and $f\in B_b(\R^d)$,
  $$P_t f(x)={\Ee^xf(X_t')}=\widetilde{{\Ee}}^{(x,y)}f(X_t')$$
and
  $$P_t f(y)={\Ee^yf(X_t'')}=\widetilde{{\Ee}}^{(x,y)}f(Y'_t).$$

\section{Proofs}\label{section2}

This section consists of two parts. We present in the first subsection the proofs of
Theorem \ref{th1.1} and Proposition \ref{pro-ex}, and then we prove Theorem \ref{f-thm} in
Subsection 3.2.

\subsection{Proof of Theorem $\ref{th1.1}$}
To prove Theorem \ref{th1.1}, we fix $\beta\in (0,\alpha_0\wedge1)$.
For any $n\ge1$, define an increasing function $f_{n}\in
C_b^2([0,\infty))$ such that $f_{n}(r)\leq r^\beta$ for all $r\geq
0$ and
  $$f_{n}(r)=
  \begin{cases}
  r^2, & 0\le r\le 1/(n+1);\\
  r^\beta, &1/n\le r\le 1.\\
  \end{cases}
  $$
The following estimate is critical for the proof of Theorem \ref{th1.1}.

\begin{proposition}\label{3-prop-0}
There exist three positive constants $C_1$, $C_2$ and $C_3$ such that for all $n\ge 1$ and for any $x$,
$y\in \R^d$ with $1/n\le |x-y|\le 1$,
  \begin{equation}\label{proof4-0}
  \begin{split}
  & \widetilde{L} f_{n}(|x-y|)\\
  &\le -\  C_1  |x-y|^{\beta-\alpha(x)\wedge\alpha(y)}+C_2\Big[1+|x-y|^{\beta-\alpha(x)\wedge\alpha(y)}A(x,y)\Big]\\
  &\hskip14pt + C_3|x-y|^{\beta-2}\int_{\left\{\frac{|x-y|}{2}<|z|<1\right\}}
  \langle
  x-y,z\rangle\bigg(\frac{n(y,z)}{|z|^{d+\alpha(y)}}-\frac{n(x,z)}{|z|^{d+\alpha(x)}}\bigg)\,dz,
  \end{split}
  \end{equation}
where
  \begin{align*} A(x,y)& :=
  |\alpha(x)-\alpha(y)|\Big(\log\frac{1}{|x-y|}\Big) |x-y|^{-|\alpha(x)-\alpha(y)|}\\
  &\hskip16pt +\Big(\sup_{|z|\leq 1} |n(x,z)-n(y,z)|
  +\sup_{z\in\R^d, |z_1-z_2|\le |x-y|}|n(z,z_1)-n(z,z_2)|\Big).
  \end{align*}
\end{proposition}

\begin{proof} By the definition \eqref{coup-1} of the coupling operator $\widetilde L$,
we shall estimate the two terms $\widetilde{L}_1f_{n}(|x-y|)$ and $\widetilde{L}_2f_{n}(|x-y|)$
separately.

(1)  First, for any $x$, $y\in \R^d$ with $1/n\le |x-y|\le 1$, by
(A3), \begin{equation}\label{f-1}\<\nabla_xf_n(|x-y|),
z+\varphi_{x,y}(z)\>=0\quad \mbox{and}\quad
  \<\nabla_yf_n(|x-y|), z+\varphi_{x,y}(z)\>=0.\end{equation}
We have for any $x$, $y\in \R^d$ with $1/n\le |x-y|\le 1$,
  \begin{align*}
  & \widetilde{L}_1f_{n}(|x-y|)\\
  &\le \frac{1}{2}\bigg[\int_{\left\{|z|\le\frac{|x-y|}{2}\right\}}\! \bigg(\Big|(x-y)+\frac{2\langle x-y,z\rangle }{|x-y|^2}(x-y)\Big|^\beta \\
  &\hskip90pt + \Big|(x-y)-\frac{2\langle x-y,z\rangle }{|x-y|^2}(x-y)\Big|^\beta \!
  -2|x-y|^\beta\bigg)\frac{\tilde{n}(x,y,z)}{|z|^{d+\alpha(x)\wedge\alpha(y)}}\,dz\bigg]\\
  &\quad +\bigg[\int_{\left\{|z|\le\frac{|x-y|}{2}\right\}} \Big(|x-y+z|^\beta-|x-y|^\beta-\beta|x-y|^{\beta-2}\langle x-y,z\rangle\Big)\\
  &\hskip90pt \times\bigg(\frac{n(x,z)}{|z|^{d+\alpha(x)}}-\frac{\tilde{n}(x,y,z)}{|z|^{d+\alpha(x)\wedge\alpha(y)}}\bigg)\,dz\\
  &\hskip30pt + \int_{\left\{|z|\le\frac{|x-y|}{2}\right\}} \Big(|x-y-z|^\beta-|x-y|^\beta+\beta|x-y|^{\beta-2}\langle x-y,z\rangle\Big)\\
  &\hskip95pt \times\bigg(\frac{n(y,z)}{|z|^{d+\alpha(y)}}
  -\frac{\tilde{n}(x,y,z)}{|z|^{d+\alpha(x)\wedge\alpha(y)}}\bigg)\,dz\bigg]\\
  &=: \widetilde{L}_{1,1}f_{n}(|x-y|)+ \widetilde{L}_{1,2}f_{n}(|x-y|).
  \end{align*}

On the one hand, using the facts that
  $$(1+s)^\beta+(1-s)^\beta\le 2+\beta(\beta-1)s^2, \quad 0\le s<1, \beta\in(0,1)$$
and $n(x,z)\ge c_1$ for all $x$, $z\in\R^d$,
we deduce that there exists a constant $c_{1,1}>0$ such that for any
$x$, $y\in \R^d$ with $1/n\le |x-y|\le 1$,
 \begin{align*}
  \widetilde{L}_{1,1}f_{n}(|x-y|)
  &\le {2}{\beta(\beta-1)}|x-y|^{\beta-2}\int_{\left\{|z|\le\frac{|x-y|}{2}\right\}}
  \frac{|\langle x-y,z\rangle|^2}{|x-y|^2}\cdot \frac{\tilde{n}(x,y,z)}{|z|^{d+\alpha(x)\wedge\alpha(y)}}\,dz\\
  &\le {2}{c_1\beta(\beta-1)}|x-y|^{\beta-2}\int_{\left\{|z|\le\frac{|x-y|}{2}\right\}}
  \frac{|\langle x-y,z\rangle|^2}{|x-y|^2}\cdot \frac{1}{|z|^{d+\alpha(x)\wedge\alpha(y)}}\,dz\\
  &= {2}{c_1\beta(\beta-1)}|x-y|^{\beta-2}\int_{\left\{|z|\le\frac{|x-y|}{2}\right\}}
  \frac{|z_1|^2}{|z|^{d+\alpha(x)\wedge\alpha(y)}}\,dz\\
  &\le -{c_{1,1}} |x-y|^{\beta-\alpha(x)\wedge\alpha(y)},
  \end{align*}
where in the equality above we have used the rotational invariance property of the kernel
$\frac{1}{|z|^{d+\alpha(x)\wedge\alpha(y)}}\,dz$, and in the last
inequality we have used that ${\alpha(x)\wedge\alpha(y)}\ge
\alpha_0>0$ for all $x$, $y\in\R^d$.

On the other hand, applying the inequality
  $$b^\beta- a^\beta\le \beta a^{\beta-1}(b-a),\quad a,b>0, \beta\in(0,1),$$
we know that for all $x$, $y$ and $z\in\R^d$,
  $$|x-y+z|^\beta-|x-y|^\beta\le \beta|x-y|^{\beta-1}\big(|x-y+z|-|x-y|\big).$$
Moreover, for all $x$, $y$ and $z\in\R^d$,
 \begin{equation}\label{f-0}
  \begin{split}
  &|x-y|\cdot|x+z-y|-|x-y|^2-\langle x-y,z\rangle\\
  &= \frac{1}{2}\Big(2|x-y|\cdot |x-y+z|-2|x-y|^2-|x+z-y|^2+|z|^2+|x-y|^2\Big)\\
  &=  \frac{1}{2}\Big(|z|^2-\big(|x-y-z|-|x-y|\big)^2\Big)
  \le \frac{|z|^2}{2}.   \end{split}\end{equation}
Thus for any $x$, $y\in \R^d$ with $1/n\le |x-y|\le 1$, we arrive at
  \begin{align*}
  \widetilde{L}_{1,2}f_{n}(|x-y|)
  &\le \frac{\beta}{2}|x-y|^{\beta-2}\int_{\left\{|z|\le\frac{|x-y|}{2}\right\}}
  |z|^2\bigg(\frac{n(x,z)}{|z|^{d+\alpha(x)}}-\frac{\tilde{n}(x,y,z)}{|z|^{d+\alpha(x)\wedge\alpha(y)}}\bigg)\,dz\\
  &\quad+ \frac{\beta}{2}|x-y|^{\beta-2}
 \int_{\left\{|z|\le\frac{|x-y|}{2}\right\}} |z|^2\bigg(\frac{n(y,z)}{|z|^{d+\alpha(y)}}
  -\frac{\tilde{n}(x,y,z)}{|z|^{d+\alpha(x)\wedge\alpha(y)}}\bigg)\,dz. \end{align*}

Furthermore, for all $x$, $y$ and $z\in\R^d$,
 \begin{equation}\label{f-2}
  \begin{split}
  &\frac{n(x,z)}{|z|^{d+\alpha(x)}}-\frac{\tilde{n}(x,y,z)}{|z|^{d+\alpha(x)\wedge\alpha(y)}}\\
  &=\bigg(\frac{n(x,z)}{|z|^{d+\alpha(x)}}-\frac{n(x,z)}{|z|^{d+\alpha(x)\wedge\alpha(y)}}\bigg)
  +\bigg(\frac{n(x,z)}{|z|^{d+\alpha(x)\wedge\alpha(y)}}-\frac{\tilde{n}(x,y,z)}{|z|^{d+\alpha(x)\wedge\alpha(y)}} \bigg).
  \end{split}\end{equation}
For the first term, we note that $n(x,z)\le c_2$ for all $x$,
$z\in\R^d$, and by the mean value theorem, for all $z\in\R^d$ with
$|z|\le 1$,
\begin{equation}\label{eee}
  \begin{split}
  &\frac{1}{|z|^{d+\alpha(x)}}-\frac{1}{|z|^{d+\alpha(x)\wedge\alpha(y)}}\\
  &=\frac{1}{|z|^{d+\alpha(x)\wedge\alpha(y)}}
  \bigg[ \Big(\frac{1}{|z|}\Big)^{\alpha(x)-\alpha(x)\wedge \alpha(y)}-1\bigg] \\
  &\le \frac{1}{|z|^{d+\alpha(x)\wedge\alpha(y)}} \Big(\frac{1}{|z|}\Big)^{\alpha(x)-\alpha(x)\wedge \alpha(y)}
  \Big(\log \frac{1}{|z|}\Big)
  \big[\alpha(x)-\alpha(x)\wedge \alpha(y)\big]\\
  &=\frac{1}{|z|^{d+\alpha(x)}} \Big(\log \frac{1}{|z|}\Big)
  \big|\alpha(x)-\alpha(y)\big|.
  \end{split}\end{equation}
The estimate of  the second term follows from the definition of the function $\tilde{n}(x,y,z)$:
   \begin{equation}\label{f-3}
  \begin{split}
  &\frac{n(x,z)}{|z|^{d+\alpha(x)\wedge\alpha(y)}}-\frac{\tilde{n}(x,y,z)}{|z|^{d+\alpha(x)\wedge\alpha(y)}}\\
  &\le \frac{1}{|z|^{d+\alpha(x)\wedge\alpha(y)}}\bigg[\sup_{ |z|\le 1} |n(x,z)-n(y,z)|
  \\
  &\qquad\qquad\qquad\qquad+\sup_{z\in\R^d, |z_1-z_2|\le |x-y|}|n(z,z_1)-n(z,z_2)|\bigg].
  \end{split}\end{equation}

Hence, we get that there exists a constant $c_{1,2}>0$ such that for any $x$, $y\in \R^d$
with $1/n\le |x-y|\le 1$,
  \begin{align*}
  & \widetilde{L}_{1,2}f_{n}(|x-y|)\\
  &\le \frac{c_2\beta}{2}|x-y|^{\beta-2}|\alpha(x)-\alpha(y)|\int_{\left\{|z|\le\frac{|x-y|}{2}\right\}}
  \frac{ |z|^2}{|z|^{d+\alpha(x)}}\Big(\log \frac{1}{|z|}\Big)\,dz\\
  &\quad+\frac{c_2\beta}{2}|x-y|^{\beta-2}|\alpha(x)-\alpha(y)|\int_{\left\{|z|\le\frac{|x-y|}{2}\right\}}
  \frac{ |z|^2}{|z|^{d+\alpha(y)}}\Big(\log \frac{1}{|z|}\Big)\,dz\\
  &\quad+ {\beta}|x-y|^{\beta-2}\bigg( \sup_{|z|\le 1} |n(x,z)-n(y,z)|
  +\sup_{z\in\R^d,|z_1-z_2|\le |x-y|}|n(z,z_1)-n(z,z_2)|\bigg)\\
  &\hskip86pt \times \int_{\left\{|z|\le\frac{|x-y|}{2}\right\}}
  \frac{|z|^2}{|z|^{d+\alpha(x)\wedge\alpha(y)}}\,dz\\
  &\le c_{1,2} |x-y|^{\beta-\alpha(x)\wedge\alpha(y)}\bigg[|\alpha(x)-\alpha(y)|\Big(\log\frac{1}{|x-y|}\Big) |x-y|^{-|\alpha(x)-\alpha(y)|}\\
  &\qquad\quad +\bigg(\sup_{|z|\le 1}
  |n(x,z)-n(y,z)|+\sup_{z\in\R^d, |z_1-z_2|\le |x-y|}|n(z,z_1)-n(z,z_2)|\bigg)\bigg], \end{align*}
where in the last inequality we have used the fact that
$\alpha(x)\wedge\alpha(y)\ge \alpha_0$ for all $x$, $y\in\R^d.$

(2) Secondly, for any $x$, $y\in \R^d$ with $1/n\le |x-y|\le 1$,
  \begin{align*}
  & \widetilde{L}_2f_{n}(|x-y|)\\
  &\le \int_{\left\{|z|>\frac{|x-y|}{2}\right\}}
  \Big(|x-y+z|^\beta-|x-y|^\beta-\beta|x-y|^{\beta-2} \langle x-y,z\rangle \I_{\{|z|\le 1\}}\Big)\\
  &\hskip75pt \times\bigg(\frac{n(x,z)}{|z|^{d+\alpha(x)}}-\frac{n(x,z)\wedge n(y,z)}{|z|^{d+\alpha(x)}\vee |z|^{d+\alpha(y)}}\bigg)\,dz\\
  &\quad+ \int_{\left\{|z|>\frac{|x-y|}{2}\right\}}
  \Big(|x-y-z|^\beta-|x-y|^\beta+\beta|x-y|^{\beta-2} \langle x-y,z\rangle\I_{\{|z|\le 1\}}\Big)\\
  &\hskip75pt \times\bigg(\frac{n(y,z)}{|z|^{d+\alpha(y)}}-\frac{n(x,z)\wedge n(y,z)}{|z|^{d+\alpha(x)}\vee |z|^{d+\alpha(y)}}\bigg)\,dz\\
   &\le \int_{\left\{|z|>\frac{|x-y|}{2}\right\}}
  |z|^\beta\bigg(\frac{n(x,z)}{|z|^{d+\alpha(x)}}-\frac{n(x,z)\wedge n(y,z)}{|z|^{d+\alpha(x)}\vee |z|^{d+\alpha(y)}}\bigg)\,dz\\
  &\quad+ \int_{\left\{|z|>\frac{|x-y|}{2}\right\}}
  |z|^\beta\bigg(\frac{n(y,z)}{|z|^{d+\alpha(y)}}-\frac{n(x,z)\wedge n(y,z)}{|z|^{d+\alpha(x)}\vee |z|^{d+\alpha(y)}}\bigg)\,dz\\
  &\quad-\beta|x-y|^{\beta-2}\int_{\left\{\frac{|x-y|}{2}<|z|<1\right\}}\langle x-y,z\rangle\bigg(\frac{n(x,z)}{|z|^{d+\alpha(x)}}-\frac{n(y,z)}{|z|^{d+\alpha(y)}}\bigg)\,dz,
  \end{align*}
where in the second inequality we have used the inequality that
$$(a+b)^\beta\le a^\beta+b^\beta,\quad a, b>0, \beta\in(0,1).$$

On the one hand, it is easy to see that
\begin{align*}
  &\int_{\left\{|z|\ge 1\right\}}
  |z|^\beta\bigg(\frac{n(x,z)}{|z|^{d+\alpha(x)}}-\frac{n(x,z)\wedge n(y,z)}{|z|^{d+\alpha(x)}\vee |z|^{d+\alpha(y)}}\bigg)\,dz\\
  &\quad+ \int_{\left\{|z|>1\right\}}
 |z|^\beta\bigg(\frac{n(y,z)}{|z|^{d+\alpha(y)}}-\frac{n(x,z)\wedge n(y,z)}{|z|^{d+\alpha(x)}\vee |z|^{d+\alpha(y)}}\bigg)\,dz\\
 & \le c_2\int_{\left\{|z|\ge 1\right\}}
 \bigg(\frac{1}{|z|^{d+\alpha(x)-\beta}}+\frac{1}{|z|^{d+\alpha(y)-\beta}}\bigg)\,dz\\
 &\le 2 c_2\int_{\left\{|z|\ge 1\right\}}
 \frac{1}{|z|^{d+\alpha_0-\beta}}\,dz=: c_{2,1}<\infty.
  \end{align*}
On the other hand, following the argument for the estimate of $
\widetilde{L}_{1,2}f_{n}(|x-y|)$, we can find a constant $c_{2,2}>0$
such that for any $x$, $y\in \R^d$ with $1/n\le |x-y|\le 1$,
  \begin{align*}
  &\int_{\left\{\frac{|x-y|}{2}<|z|\le1\right\}}
  |z|^\beta\bigg(\frac{n(x,z)}{|z|^{d+\alpha(x)}}-\frac{n(x,z)\wedge n(y,z)}{|z|^{d+\alpha(x)}\vee |z|^{d+\alpha(y)}}\bigg)\,dz\\
  &\quad+ \int_{\left\{\frac{|x-y|}{2}<|z|\le1\right\}}
 |z|^\beta\bigg(\frac{n(y,z)}{|z|^{d+\alpha(y)}}-\frac{n(x,z)\wedge n(y,z)}{|z|^{d+\alpha(x)}\vee |z|^{d+\alpha(y)}}\bigg)\,dz\\
  &\le c_{2,2} |x-y|^{\beta-\alpha(x)\wedge\alpha(y)}\bigg[|\alpha(x)-\alpha(y)|\Big(\log\frac{1}{|x-y|}\Big) |x-y|^{-|\alpha(x)-\alpha(y)|}\\
  &\qquad\qquad\qquad\qquad \qquad+\Big(\sup_{|z|\le 1} |n(x,z)-n(y,z)|\Big)\bigg]. \end{align*}
Thus, for any $x$, $y\in \R^d$ with $1/n\le |x-y|\le 1$,
  \begin{align*}
  & \widetilde{L}_2f_{n}(|x-y|)\\
  &\le c_{2,1}+ c_{2,2} |x-y|^{\beta-\alpha(x)\wedge\alpha(y)}\bigg[|\alpha(x)-\alpha(y)|\Big(\log\frac{1}{|x-y|}\Big) |x-y|^{-|\alpha(x)-\alpha(y)|}\\
  &\qquad\qquad\qquad\qquad \qquad\qquad\quad+\Big(\sup_{|z|\le 1} |n(x,z)-n(y,z)|\Big)\bigg]\\
 &\quad-\beta|x-y|^{\beta-2}\int_{\left\{\frac{|x-y|}{2}<|z|<1\right\}}\langle x-y,z\rangle\bigg(\frac{n(x,z)}{|z|^{d+\alpha(x)}}-\frac{n(y,z)}{|z|^{d+\alpha(y)}}\bigg)\,dz
  \end{align*}

(3) The required assertion \eqref{proof4-0} immediately follows from  the two
above estimates on $\widetilde{L}_1f_{n}(|x-y|)$ and
$\widetilde{L}_2f_{n}(|x-y|)$.
\end{proof}

\begin{corollary}\label{cor-1}
Under assumptions ${\rm (H2)}$, ${\rm (H3)}$ and ${\rm (H4)}$, there
exist two constants $\varepsilon_0:=\varepsilon_0(\beta,\alpha(x),
n(x,z))\in (0,1)$ and $C_0:=C_0(\beta,\alpha(x), n(x,z))>0$, such
that for all $n\ge1$, $\varepsilon\in(0,\varepsilon_0]$ and any $x$,
$y\in \R^d$ with $1/n\le |x-y|< \varepsilon$,
  \begin{equation}\label{proof4}
  \begin{split}
 \widetilde{L} f_{n}(|x-y|)\le -{C_0\, \varepsilon^{\beta-\alpha_0}}=:A_{\beta,\alpha}<0.
  \end{split}
  \end{equation}

\end{corollary}

\begin{proof} By \eqref{h2},
\begin{equation}\label{ppp}\lim_{|x-y|\to 0} |\alpha(x)-\alpha(y)|\Big(\log\frac{1}{|x-y|}\Big)
|x-y|^{-|\alpha(x)-\alpha(y)|}=0.\end{equation} Due to \eqref{h3},
$$\lim_{|x-y|\to0} \Big(\sup_{|z|\le 1} |n(x,z)-n(y,z)|+\sup_{z\in\R^d, |z_1-z_2|\le
|x-y|}|n(z,z_1)-n(z,z_2)|\Big)=0.$$ Noticing that for all $x$, $y\in\R^d$ with $|x-y|\le 1$,
$$|x-y|^{\beta-\alpha(x)\wedge\alpha(y)}\ge
|x-y|^{\beta-\alpha_0},$$ we obtain from \eqref{h3} and
\eqref{proof4-0} that there is a constant
$\varepsilon_0:=\varepsilon_0(\beta,\alpha(x), n(x,z))\in (0,1)$
such that for all $\varepsilon\in(0,\varepsilon_0]$ and any $x$,
$y\in \R^d$ with $1/n\le |x-y|\le \varepsilon$,
  $$
  \widetilde{L} f_{n}(|x-y|)\le -\frac{C_1}{2}  |x-y|^{\beta-\alpha(x)\wedge\alpha(y)}\le -\frac{C_1\varepsilon^{\beta-\alpha_0}}{2}.$$
This proves the desired assertion.\end{proof}

Now we are ready to present the

\begin{proof}[Proof of Theorem $\ref{th1.1}$]
We will use the coupling process $(X'_t,Y'_t)_{t\ge0}$ constructed
in Section \ref{section22}. Denote by $\widetilde{\Pp}^{(x,y)}$ and
$\widetilde{\Ee}^{(x,y)}$ the distribution and the expectation of
$(X'_t,Y'_t)_{t\ge0}$ starting from $(x,y)$, respectively. Recall
that $\varepsilon_0>0$ is given in Corollary \ref{cor-1}. For any
$n\ge1$ and $\varepsilon\in(0,\varepsilon_0]$, we set
  $$\aligned S_\varepsilon&:=\inf\{t\ge0: |X'_t-Y'_t|>\varepsilon\},\\
  T_n&:=\inf\{t\ge0: |X'_t-Y'_t|\le 1/n\},\\
  T_{n,\varepsilon}&:=T_n\wedge S_\varepsilon.\endaligned$$
Furthermore, we will still use the coupling time defined by
  $$T:=\inf\{t\ge0: X'_t=Y'_t\}.$$
Note that $T_n\uparrow T$ as $n\uparrow\infty.$

For any $x,$ $y\in\R^d$ with $0<|x-y|<\varepsilon\le \varepsilon_0$,
we take $n$ large enough such that $|x-y|>1/n$. Then, by
\eqref{proof4},
  $$\aligned
  0&\le\widetilde{\Ee}^{(x,y)}f_{n}\big(|X'_{t\wedge T_{n,\varepsilon}}-Y'_{t\wedge T_{n,\varepsilon}}|\big)\\
  &=f_{n}(|x-y|)+\widetilde{\Ee}^{(x,y)}\bigg(\int_0^{t\wedge T_{n,\varepsilon}} \widetilde{L} f_{n}\big(|X'_{s}-Y'_{s}|\big)\,ds\bigg)\\
  &\le f_{n}(|x-y|)+A_{\beta,\alpha}\widetilde{\Ee}^{(x,y)}(t\wedge
  T_{n,\varepsilon}).\endaligned$$
Therefore
  $$\widetilde{\Ee}^{(x,y)}(t\wedge T_{n,\varepsilon})\le
  -\frac{ f_{n}(|x-y|)}{ A_{\beta,\alpha}}=-\frac{ |x-y|^\beta}{ A_{\beta,\alpha}}.$$
Letting $t\rightarrow\infty$ and then $n\rightarrow\infty$, we
arrive at
  $$\widetilde{\Ee}^{(x,y)}(T_{}\wedge S_\varepsilon)
  \le -\frac{ |x-y|^\beta}{ A_{\beta,\alpha}}.$$

On the other hand, again by \eqref{proof4}, for any $x$, $y\in\R^d$
with $1/n\le |x-y|<\varepsilon\le \varepsilon_0$,
  $$\aligned \widetilde{\Ee}^{(x,y)}f_{n}\big(|X'_{t\wedge T_{n,\varepsilon}}&-Y'_{t\wedge T_{n,\varepsilon}}|\big)\\
  &=f_{n}(|x-y|)+\widetilde{\Ee}^{(x,y)}\bigg(\int_0^{t\wedge T_{n,\varepsilon}}
  \widetilde{L} f_{n}(| X'_u-Y'_u|)\,du\bigg)\\
  &\le f_{n}(|x-y|).\endaligned$$
This yields that
  $$f_{n}(\varepsilon)\widetilde{\Pp}^{(x,y)}(S_\varepsilon<T_n\wedge t)\le f_{n}(|x-y|).$$
Letting $t\rightarrow\infty$ and then $n\to \infty$ leads to
  $$\widetilde{\Pp}^{(x,y)}(T>S_\varepsilon)\le\frac{|x-y|^\beta}{\varepsilon^\beta}.$$

Therefore, for any $x$, $y\in\R^d$ with $0<|x-y|<\varepsilon\le
\varepsilon_0$,
  $$\aligned
   \widetilde{\Pp}^{(x,y)}(T\ge t)
    &\le \widetilde{\Pp}^{{(x,y)}}(T\wedge S_\varepsilon>t)+\widetilde{\Pp}^{{(x,y)}}(T>S_\varepsilon)\\
  &\le \frac{\widetilde{\Ee}^{(x,y)}(T\wedge S_\varepsilon)}{t}+\frac{|x-y|^\beta}{\varepsilon^\beta}\\
  &\le -\frac{ |x-y|^\beta}{t A_{\beta,\alpha}}+\frac{|x-y|^\beta}{\varepsilon^\beta}\\
  &=|x-y|^\beta\bigg[\frac{1}{\varepsilon^\beta}-\frac{1}{t A_{\beta,\alpha}}\bigg].
  \endaligned$$
Hence, for any $f\in B_b(\R^d)$, $t>0$ and any $x$, $y\in\R^d$ with
$0<|x-y|<\varepsilon\le \varepsilon_0$,
  \begin{align*}
  {|P_t f(x)-P_t f(y)|}&={|\Ee^xf(X_t')-\Ee^yf(Y'_t)|}\\
  &={\big|\widetilde{{\Ee}}^{(x,y)}(f(X_t')-f(Y'_t))\big|}\\
  &={\big|\widetilde{{\Ee}}^{(x,y)}(f(X_t')-f(Y'_t))\I_{\{T\ge t\}}\big|}\\
  &\le\|f\|_\infty {\widetilde{{\Pp}}^{(x,y)}(T\ge t)}\\
  &\le\|f\|_\infty|x-y|^\beta\bigg[\frac{1}{\varepsilon^\beta}-\frac{1}{t A_{\beta,\alpha}}\bigg],
  \end{align*}
which immediately yields that
  $$\sup_{|x-y|\le\varepsilon}\frac{{|P_t f(x)-P_t f(y)|}}{|x-y|^\beta}
  \le \|f\|_\infty\bigg[\frac{1}{\varepsilon^\beta}-\frac{1}{t A_{\beta,\alpha}}\bigg].$$

This along with the fact that
  $$\sup_{|x-y|\ge\varepsilon}\frac{{|P_t f(x)-P_t f(y)|}}{|x-y|^\beta}
  \le 2\|f\|_\infty\, \varepsilon^{-\beta}$$
further gives us that for all $\varepsilon\in(0,\varepsilon_0]$,
  $$\sup_{x\neq y}\frac{{|P_t f(x)-P_t f(y)|}}{|x-y|^\beta}\le
  2\|f\|_\infty \bigg[\frac{1}{\varepsilon^\beta}-\frac{1}{t A_{\beta,\alpha}}\bigg]
  = 2\|f\|_\infty\bigg[\frac{1}{\varepsilon^\beta}+\frac{\varepsilon^{\alpha_0-\beta}}{C_0t}\bigg], $$
where $C_0:=C_0(\beta, \alpha(x),n(x,z))$ is a positive
constant independent of $\varepsilon$, e.g.\ see \eqref{proof4}. In
particular, we have
  \begin{equation}\label{zhangxc}
  \sup_{x\neq y}\frac{{|P_t f(x)-P_t f(y)|}}{|x-y|^\beta}
  \le 2\|f\|_\infty\inf_{\varepsilon\in(0,\varepsilon_0]}
  \bigg[\frac{1}{\varepsilon^\beta}+\frac{\varepsilon^{\alpha_0-\beta}}{C_0t}\bigg].
  \end{equation}
Therefore, we complete the proof by taking
$\varepsilon=t^{1/\alpha_0}\wedge \varepsilon_0$ in \eqref{zhangxc}.
\end{proof}

At the end of this subsection, we would like to present the

\begin{proof}[Proof of Proposition $\ref{pro-ex}$] (1) For all $z\in\R^d$ with $|z|\le 1$,
  \begin{equation}\label{rrr}
  \begin{split}
  &\bigg|\int_{\left\{\frac{|x-y|}{2}<|z|\le1\right\}}\langle x-y,z\rangle\bigg(\frac{n(y,z)}{|z|^{d+\alpha(y)}}-\frac{n(x,z)}{|z|^{d+\alpha(x)}}\bigg)\,dz\bigg|\\
  &\le |x-y|\int_{\left\{\frac{|x-y|}{2}<|z|\le1\right\}}|z|\bigg|\frac{n(y,z)}{|z|^{d+\alpha(y)}}
  -\frac{n(x,z)}{|z|^{d+\alpha(x)}}\bigg| \,dz\\
  &\le |x-y|\int_{\left\{\frac{|x-y|}{2}<|z|\le1\right\}}|z|\bigg(\frac{n(y,z)}{|z|^{d+\alpha(y)}}
  -\frac{n(y,z)}{|z|^{d+\alpha(x)\wedge\alpha(y)}}\bigg)\,dz\\
  &\quad+|x-y|\int_{\left\{\frac{|x-y|}{2}<|z|\le1\right\}}|z|\bigg(\frac{n(x,z)}{|z|^{d+\alpha(x)}}
  -\frac{n(x,z)}{|z|^{d+\alpha(x)\wedge\alpha(y)}}\bigg)\,dz\\  &\quad+|x-y|\int_{\left\{\frac{|x-y|}{2}<|z|<1\right\}}|z|\frac{|n(x,z)-n(y,z)|}{|z|^{d+\alpha(x)\wedge\alpha(y)}}\,dz.
  \end{split}\end{equation}

On the one hand, there exists a constant $c_{1,1}>0$ such that for
all $x$, $y\in\R^d$,
  \begin{align*}
  &\int_{\left\{\frac{|x-y|}{2}<|z|\le1\right\}}|z|\frac{|n(x,z)-n(y,z)|}{|z|^{d+\alpha(x)\wedge\alpha(y)}}\,dz\\
  &\le \Big(\sup_{|z|\le 1} |n(x,z)-n(y,z)|\Big) \int_{\left\{\frac{|x-y|}{2}<|z|\le1\right\}}\frac{|z|}{|z|^{d+\alpha(x)\wedge\alpha(y)}}\,dz\\
  &\le c_{1,1}|x-y|^{1-\alpha(x)\wedge\alpha(y)}\Big(\sup_{|z|\le 1} |n(x,z)-n(y,z)|\Big),
  \end{align*}
where in the last inequality we have used the fact that
$\inf_{x\in\R^d}\alpha(x)>1$. On the other hand, by \eqref{eee},
there exists a constant $c_{1,2}>0$ such that for all $x$,
$y\in\R^d$,
  \begin{align*} &\int_{\left\{\frac{|x-y|}{2}<|z|\le1\right\}}|z|\bigg(\frac{n(x,z)}{|z|^{d+\alpha(x)}}
  -\frac{n(x,z)}{|z|^{d+\alpha(x)\wedge\alpha(y)}}\bigg)\,dz\\
  &\le c_2 |\alpha(x)-\alpha(y)|\int_{\left\{\frac{|x-y|}{2}<|z|\le1\right\}}\frac{|z|}{|z|^{d+\alpha(x)}}\Big(\log \frac{1}{|z|}\Big)\,dz\\
  &\le c_{1,2 }|\alpha(x)-\alpha(y)| |x-y|^{1-\alpha(x)}\log \frac{1}{|x-y|}\\
  &\le c_{1,2} |\alpha(x)-\alpha(y)| |x-y|^{1-\alpha(x)\wedge\alpha(y)}
  |x-y|^{-|\alpha(x)-\alpha(y)|}\log \frac{1}{|x-y|}, \end{align*}
where in the second inequality we have used again that
$\inf_{x\in\R^d} \alpha(x)>1$. Similarly, we have that
  \begin{align*} &\int_{\left\{\frac{|x-y|}{2}<|z|<1\right\}}|z|\bigg(\frac{n(y,z)}{|z|^{d+\alpha(y)}}
  -\frac{n(y,z)}{|z|^{d+\alpha(x)\wedge\alpha(y)}}\bigg)\,dz\\
  &\le c_{1,3} |\alpha(x)-\alpha(y)||x-y|^{-\alpha(x)\wedge\alpha(y)+1}
  |x-y|^{-|\alpha(x)-\alpha(y)|}\log \frac{1}{|x-y|}
  \end{align*}
holds for some constant $c_{1,3}>0$ and all $x$, $y\in\R^d$.

Combining all the conclusions above with \eqref{h3} and \eqref{ppp},
we get \eqref{h4}.

\smallskip

(2) We mainly follow the arguments of part (1), and here we only
sketch the proof.

Since $\sup_{x\in\R^d} \alpha(x)<1$, there exist constants
$c_{2,1}$, $c_{2,2}>0$ such that for all $x$, $y\in\R^d$,
 \begin{align*}&\int_{\left\{\frac{|x-y|}{2}<|z|\le1\right\}}|z|\frac{|n(x,z)-n(y,z)|}{|z|^{d+\alpha(x)\wedge\alpha(y)}}\,dz\le c_{2,1}\Big(\sup_{|z|\le 1} |n(x,z)-n(y,z)|\Big) ,\end{align*}
  and \begin{align*} &\int_{\left\{\frac{|x-y|}{2}<|z|\le1\right\}}|z|\bigg(\frac{n(x,z)}{|z|^{d+\alpha(x)}}
  -\frac{n(x,z)}{|z|^{d+\alpha(x)\wedge\alpha(y)}}\bigg)\,dz
\le c_{2,2} |\alpha(x)-\alpha(y)|. \end{align*}
This, along with \eqref{rrr} and \eqref{ttt}, gives us the required assertion.
\end{proof}

\subsection{Proof of Theorem \ref{f-thm}}
In this part, suppose that $\alpha_0=\inf_{x\in\R^d} \alpha(x)>1$.
We first give an elementary lemma.

\begin{lemma}\label{3-lem-1}
Let $\varphi\in \mathscr{D}$ and $f(r)=\int_0^r\varphi(s)\,ds$ for all $r\in(0,2]$.
Then, for any $0\le \delta\le a\leq 1$,
  $$f(a+\delta)+f(a-\delta)-2f(a)\leq \varphi'(2a)\delta^2.$$
\end{lemma}

\begin{proof} For any $0\le \delta\le a$,
  \begin{align*}
  f(a+\delta)+f(a-\delta)-2f(a)&=\int_a^{a+\delta}d s\int_{s-\delta}^s \varphi'(r)\,d r\\
    &\leq \varphi'(a+\delta)\delta^2\leq \varphi'(2a)\delta^2,
  \end{align*}
where in the two inequalities we have used the fact that $\varphi''>0$.
\end{proof}

Let $\varphi\in \mathscr{D}$ and $f$ be the function defined in Lemma \ref{3-lem-1}. For any $n\ge1$, define an increasing
function $f_{n}\in C_b^2([0,\infty))$ such that $ f_{n}(r)\le
f(r)$ for all $0<r\le 2$, and
  $$f_{n}(r)=
  \begin{cases}
  f(r), &1/n\le r\le2;\\
  f(2)+1,&r\ge 3.\\
  \end{cases}
  $$
The next result is analogous to Proposition \ref{3-prop-0}.

\begin{proposition}\label{3-prop-1}
There exist three positive constants $C_1$, $C_2$ and $C_3$
such that for all $n\ge 1$ and for any $x$,
$y\in \R^d$ with $1/n\le |x-y|\le 1$,
  \begin{equation*}
  \begin{split}
  & \widetilde{L} f_{n}(|x-y|)\\
  &\le C_1\varphi'(2|x-y|)|x-y|^{2-\alpha(x)\wedge\alpha(y)}+C_2\\
  &\hskip13pt + C_3 \varphi(|x-y|) |x-y|^{1-\alpha(x)\wedge\alpha(y)}\bigg[|\alpha(x)-\alpha(y)|\Big(\log\frac{1}{|x-y|}\Big) |x-y|^{-|\alpha(x)-\alpha(y)|}\\
  &\qquad\quad +\bigg(\sup_{|z|\le 1}
  |n(x,z)-n(y,z)|+\sup_{z\in\R^d,|z_1-z_2|\le |x-y|}|n(z,z_1)-n(z,z_2)|\bigg)\bigg].
  \end{split}
  \end{equation*}
\end{proposition}

\begin{proof} We follow the line of arguments for proving Proposition \ref{3-prop-0}.

(1)  First, for any $x$, $y\in \R^d$ with $1/n\le |x-y|\le 1$, by \eqref{f-1} and
the facts that $f_n(r)\le f(r)$ for all $r\in(0,2]$ and $f_n(r)=f(r)$ for all $r\in[1/n,2]$,
  \begin{align*}
  & \widetilde{L}_1 f_n(|x-y|)\\
  &\le \frac{1}{2}\bigg[\int_{\left\{|z|\le\frac{|x-y|}{2}\right\}}\!
  \Big(f\big(\big|x-y+(z-\varphi_{x,y}(z))\big|\big)+f\big(\big|x-y-(z-\varphi_{x,y}(z))\big|\big)\\
  &\hskip90pt -2f(|x-y|)\Big)\frac{\tilde{n}(x,y,z)}{|z|^{d+\alpha(x)\wedge\alpha(y)}}\,dz\bigg]\\
  &\quad +\bigg[\int_{\left\{|z|\le\frac{|x-y|}{2}\right\}} \Big(f(|x-y+z|)-f(|x-y|)
  -f'(|x-y|)\frac{\< x-y,z\>}{|x-y|}\Big)\\
  &\hskip90pt \times\bigg(\frac{n(x,z)}{|z|^{d+\alpha(x)}}-\frac{\tilde{n}(x,y,z)}{|z|^{d+\alpha(x)\wedge\alpha(y)}}\bigg)\,dz\\
  &\hskip30pt + \int_{\left\{|z|\le\frac{|x-y|}{2}\right\}} \Big(f(|x-y-z|)-f(|x-y|)
  +f'(|x-y|)\frac{\< x-y,z\>}{|x-y|} \Big)\\
  &\hskip95pt \times\bigg(\frac{n(y,z)}{|z|^{d+\alpha(y)}}
  -\frac{\tilde{n}(x,y,z)}{|z|^{d+\alpha(x)\wedge\alpha(y)}}\bigg)\,dz\bigg]\\
  &=: \widetilde{L}_{1,1}f(|x-y|)+ \widetilde{L}_{1,2}f(|x-y|).
  \end{align*}

According to the definition of $\varphi_{x,y}(z)$ and Lemma \ref{3-lem-1}, for all $x,y,z\in\R^d$ with $|x-y|\le 1$ and $|z|\le |x-y|/2$, we have
  \begin{align*}
  &f\big(\big|x-y+(z-\varphi_{x,y}(z))\big|\big)+f\big(\big|x-y-(z-\varphi_{x,y}(z))\big|\big)-2f(|x-y|)\\
  &\leq 4\varphi'(2|x-y|)\frac{\<x-y,z\>^2}{|x-y|^2}.
  \end{align*}
As a result, since $n(x,z)\geq c_1$ for all $x$, $z\in\R^d$, there exists a constant
$c_{1,1}>0$ such that for all $x$, $y\in\R^d$ with $1/n\le |x-y|\le
1$,
  \begin{equation}\label{L11}
  \begin{split}
  \widetilde{L}_{1,1}f(|x-y|)
  &\leq 4\varphi'(2|x-y|)
  \int_{\left\{|z|\le\frac{|x-y|}{2}\right\}}\frac{\<x-y,z\>^2}{|x-y|^2}
  \cdot\frac{\tilde{n}(x,y,z)}{|z|^{d+\alpha(x)\wedge\alpha(y)}}\,dz\\
  &\leq 4c_1\varphi'(2|x-y|)\int_{\left\{|z|\le\frac{|x-y|}{2}\right\}}
  \frac{\<x-y,z\>^2}{|x-y|^2|z|^{d+\alpha(x)\wedge\alpha(y)}}\,dz\\
  &\leq c_{1,1}\varphi'(2|x-y|)|x-y|^{2-\alpha(x)\wedge\alpha(y)},
  \end{split}
  \end{equation} where in the last inequality we have used the fact that $\alpha(x)\wedge\alpha(y)\ge \alpha_0$ for all $x$, $y\in\R^d$.

Next we estimate the term $\widetilde{L}_{1,2}f(|x-y|)$. Since $\varphi'(r)<0$ for all $r\in (0,2]$,
  \begin{equation}\label{ooo}
  f(b)-f(a)\le \varphi(a)(b-a), \quad a, b\in (0,2].
  \end{equation}
We have for any $x,y, z\in\R^d$ with $1/n\leq
|x-y|\leq 1$ and $|z|\le |x-y|/2$ that
  \begin{align*}
  &f(|x-y+z|)-f(|x-y|)-f'(|x-y|)\frac{\< x-y,z\>}{|x-y|}\\
  &\leq \varphi(|x-y|)\left[(|x-y+z|-|x-y|) -\frac{\< x-y,z\>}{|x-y|}\right].
  \end{align*}
 Thus, by \eqref{f-0}, for any $x$, $y\in \R^d$ with
$1/n\le |x-y|\le 1$, we arrive at
  \begin{align*}
  \widetilde{L}_{1,2}f_n(|x-y|)
  &\le \frac{\varphi(|x-y|)}{2|x-y|}\int_{\left\{|z|\le\frac{|x-y|}{2}\right\}}
  |z|^2\bigg(\frac{n(x,z)}{|z|^{d+\alpha(x)}}-\frac{\tilde{n}(x,y,z)}{|z|^{d+\alpha(x)\wedge\alpha(y)}}\bigg)\,dz\\
  &\quad+ \frac{\varphi(|x-y|)}{2|x-y|}
  \int_{\left\{|z|\le\frac{|x-y|}{2}\right\}} |z|^2\bigg(\frac{n(y,z)}{|z|^{d+\alpha(y)}}
  -\frac{\tilde{n}(x,y,z)}{|z|^{d+\alpha(x)\wedge\alpha(y)}}\bigg)\,dz.
  \end{align*}
This, along with \eqref{f-2}, \eqref{eee},  \eqref{f-3} and the fact
that $n(x,z)\le c_2$ for all $x$, $z\in\R^d$, yields the existence of a
constant $c_{1,2}>0$ such that for all $x$, $y\in \R^d$ with $1/n\le |x-y|\le 1$,
  \begin{align*}
  & \widetilde{L}_{1,2}f(|x-y|)\\
  &\le \frac{c_2|\alpha(x)-\alpha(y)|\varphi(|x-y|)}{2|x-y|}\int_{\left\{|z|\le\frac{|x-y|}{2}\right\}}
\bigg[\frac{|z|^2}{|z|^{d+\alpha(x)}}+\frac{|z|^2}{|z|^{d+\alpha(y)}}\bigg]\Big(\log \frac{1}{|z|}\Big)\,dz\\
  &\hskip13pt + \frac{\varphi(|x-y|)}{|x-y|}\bigg( \sup_{|z|\le 1} |n(x,z)-n(y,z)|
  +\sup_{z\in\R^d,|z_1-z_2|\le |x-y|}|n(z,z_1)-n(z,z_2)|\bigg)\\
  &\hskip80pt \times \int_{\left\{|z|\le\frac{|x-y|}{2}\right\}}
  \frac{|z|^2}{|z|^{d+\alpha(x)\wedge\alpha(y)}}\,dz\\
  &\le c_{1,2} |x-y|^{1-\alpha(x)\wedge\alpha(y)}\varphi(|x-y|)\bigg[|\alpha(x)-\alpha(y)|\Big(\log\frac{1}{|x-y|}\Big) |x-y|^{-|\alpha(x)-\alpha(y)|}\\
  &\qquad\quad +\bigg(\sup_{|z|\le 1}
  |n(x,z)-n(y,z)|+\sup_{z\in\R^d,|z_1-z_2|\le |x-y|}|n(z,z_1)-n(z,z_2)|\bigg)\bigg],
  \end{align*}
where in the last inequality we have used the fact that
$\alpha(x)\wedge\alpha(y)\ge \alpha_0$ for all $x$, $y\in\R^d.$

(2) Secondly, for any $x$, $y\in \R^d$ with $1/n\le |x-y|\le 1$,
  \begin{align*}
  & \widetilde{L}_2 f_n(|x-y|)\\
  &= \int_{\left\{|z|>\frac{|x-y|}{2}\right\}}
  \Big(f_n(|x-y+z|)-f_n(|x-y|)-f'_n(|x-y|)\frac{\< x-y,z\>}{|x-y|} \I_{\{|z|\le 1\}}\Big)\\
  &\hskip75pt \times\bigg(\frac{n(x,z)}{|z|^{d+\alpha(x)}}-\frac{n(x,z)\wedge n(y,z)}{|z|^{d+\alpha(x)}\vee |z|^{d+\alpha(y)}}\bigg)\,dz\\
  &\quad+ \int_{\left\{|z|>\frac{|x-y|}{2}\right\}}
  \Big(f_n(|x-y-z|)-f_n(|x-y|)+f'_n(|x-y|)\frac{\< x-y,z\>}{|x-y|} \I_{\{|z|\le 1\}}\Big)\\
  &\hskip75pt \times\bigg(\frac{n(y,z)}{|z|^{d+\alpha(y)}}-\frac{n(x,z)\wedge n(y,z)}{|z|^{d+\alpha(x)}\vee |z|^{d+\alpha(y)}}\bigg)\,dz.
  \end{align*}
We separate each one of the above two integrals into two parts: one
part with integral domain $\{|z|>1\}$ and the other with
$\big\{\frac{|x-y|}2\leq |z|\leq 1\big\}$, and then we sum the
corresponding terms with the same integral domain. In this way we
get two quantities $I_1$ and $I_2$:
  \begin{align*}
  I_1&=\int_{\{|z|>1\}}\Big(f_n(|x-y+z|)-f_n(|x-y|)\Big)\bigg(\frac{n(x,z)}{|z|^{d+\alpha(x)}}
  -\frac{n(x,z)\wedge n(y,z)}{|z|^{d+\alpha(x)}\vee |z|^{d+\alpha(y)}}\bigg)\,dz\\
  &\hskip13pt +\int_{\{|z|>1\}}\Big(f_n(|x-y-z|)-f_n(|x-y|)\Big)\bigg(\frac{n(y,z)}{|z|^{d+\alpha(y)}}
  -\frac{n(x,z)\wedge n(y,z)}{|z|^{d+\alpha(x)}\vee |z|^{d+\alpha(y)}}\bigg)\,dz,\\
  I_2&=\int_{\left\{\frac{|x-y|}{2}<|z|\leq 1\right\}}\Big(f_n(|x-y+z|)-f_n(|x-y|)-f'_n(|x-y|)\frac{\< x-y,z\>}{|x-y|}\Big)\\
  &\hskip80pt \times\bigg(\frac{n(x,z)}{|z|^{d+\alpha(x)}}-\frac{n(x,z)\wedge n(y,z)}{|z|^{d+\alpha(x)}\vee |z|^{d+\alpha(y)}}\bigg)\,dz\\
  &\hskip13pt +\int_{\left\{\frac{|x-y|}{2}<|z|\leq 1\right\}}\Big(f_n(|x-y-z|)-f_n(|x-y|)
  +f'_n(|x-y|)\frac{\< x-y,z\>}{|x-y|}\Big)\\
  &\hskip85pt \times\bigg(\frac{n(y,z)}{|z|^{d+\alpha(y)}}-\frac{n(x,z)\wedge n(y,z)}{|z|^{d+\alpha(x)}\vee |z|^{d+\alpha(y)}}\bigg)\,dz.
  \end{align*}

On the one hand, noticing that $f_n(r)\le f(2)+1$ for all $r>0$ and $n(x,z)\le c_2$ for all $x$, $z\in\R^d$, we have
  \begin{align*}
  I_1&\leq (f(2)+1)\int_{\left\{|z|>1\right\}} \bigg(\frac{n(x,z)}{|z|^{d+\alpha(x)}}
  -\frac{n(x,z)\wedge n(y,z)}{|z|^{d+\alpha(x)}\vee |z|^{d+\alpha(y)}}\bigg)\,dz\\
  &\quad+ (f(2)+1)\int_{\left\{|z|>1\right\}} \bigg(\frac{n(y,z)}{|z|^{d+\alpha(y)}}
  -\frac{n(x,z)\wedge n(y,z)}{|z|^{d+\alpha(x)}\vee |z|^{d+\alpha(y)}}\bigg)\,dz\\
  & \le c_2(f(2)+1)\int_{\left\{|z|\ge 1\right\}}
  \bigg(\frac{1}{|z|^{d+\alpha(x)}}+\frac{1}{|z|^{d+\alpha(y)}}\bigg)\,dz\\
  &\le  c_2(f(2)+1)\int_{\left\{|z|\ge 1\right\}}
  \bigg(\frac{1}{|z|^{d+\alpha_0}}+\frac{1}{|z|^{d+\alpha_0}}\bigg)\,dz=: c_{2,1}<\infty.
  \end{align*}
On the other hand, by the definition of $f_n$ and \eqref{ooo}, for all $x,y,z\in\R^d$ with
$1/n\leq |x-y|\le 1$ and $|z|\le 1$,
  \begin{align*}
  &f_n(|x-y+z|)-f_n(|x-y|)-f'_n(|x-y|)\frac{\< x-y,z\>}{|x-y|}\\
  &\le f(|x-y+z|)-f(|x-y|)-f'(|x-y|)\frac{\< x-y,z\>}{|x-y|}\\
  &\leq \varphi(|x-y|)\Big(|x-y+z|-|x-y|-\frac{\< x-y,z\>}{|x-y|}\Big)\\
  &\le 2\varphi(|x-y|)
  |z|,
  \end{align*} where in the last inequality we have used the
  fact that $\varphi(r)>0$ for all $r\in (0,2]$.
Therefore, following the argument of
$\widetilde{L}_{1,2}f(|x-y|)$, we get that for any $x$, $y\in
\R^d$ with $1/n\le |x-y|\le 1$,
  \begin{align*}
  I_2&\leq 2 \varphi(|x-y|)\int_{\left\{\frac{|x-y|}{2}<|z|\le1\right\}}
  |z|\bigg(\frac{n(x,z)}{|z|^{d+\alpha(x)}}-\frac{n(x,z)\wedge n(y,z)}{|z|^{d+\alpha(x)}\vee |z|^{d+\alpha(y)}}\bigg)\,dz\\
  &\quad+ 2  \varphi(|x-y|)\int_{\left\{\frac{|x-y|}{2}<|z|\le1\right\}}
  |z|\bigg(\frac{n(y,z)}{|z|^{d+\alpha(y)}}-\frac{n(x,z)\wedge n(y,z)}{|z|^{d+\alpha(x)}\vee |z|^{d+\alpha(y)}}\bigg)\,dz\\
  &\le c_{2,2}  \varphi(|x-y|)|x-y|^{1-\alpha(x)\wedge\alpha(y)}\bigg[|\alpha(x)-\alpha(y)|\Big(\log\frac{1}{|x-y|}\Big) |x-y|^{-|\alpha(x)-\alpha(y)|}\\
  &\hskip122pt  +\bigg(\sup_{|z|\le 1}
  |n(x,z)-n(y,z)|\bigg)\bigg],
  \end{align*} where in the last inequality we have used the fact that $\alpha(x)\wedge\alpha(y)\ge \alpha_0>1$ for all $x$, $y\in\R^d$.

Combining the estimates on $I_1$ and $I_2$, we obtain for any $x$,
$y\in \R^d$ with $1/n\leq |x-y|\le 1$,
  \begin{align*}
  \widetilde{L}_2 f_n(|x-y|)
  &\le c_{2,1}+c_{2,2}  \varphi(|x-y|)|x-y|^{1-\alpha(x)\wedge\alpha(y)}\bigg[\sup_{|z|\le 1} |n(x,z)-n(y,z)|\\
  &\hskip60pt +|\alpha(x)-\alpha(y)|\Big(\log\frac{1}{|x-y|}\Big) |x-y|^{-|\alpha(x)-\alpha(y)|}\bigg].
  \end{align*}

(3) The required assertion immediately follows from  the two above
estimates on $\widetilde{L}_1f_n(|x-y|)$ and
$\widetilde{L}_2f_n(|x-y|)$.
\end{proof}

As a consequence of Proposition \ref{3-prop-1}, we have
\begin{corollary}\label{f-cor-1}
If \eqref{f-cor-2} holds for some $\varphi\in \mathscr{D}_{\alpha_0}$,
then there exist constants
$\varepsilon_0\in (0,1)$ and $C_0>0$ such that for all $n\ge1$,
$\varepsilon\in(0,\varepsilon_0]$ and any $x$, $y\in \R^d$ with
$1/n\le |x-y|< \varepsilon$,
  \begin{equation}\label{f-cor-3}
  \begin{split}
 \widetilde{L} f_{n}(|x-y|)\le C_0\varphi'(2\varepsilon) \varepsilon^{2-\alpha_0}<0.
  \end{split}
  \end{equation}

\end{corollary}
\begin{proof}
Due to $\varphi\in \mathscr{D}_{\alpha_0}$,
  \begin{equation}\label{iii}
  \lim_{r\to0} \varphi'(r)r^{2-\alpha_0}=-\infty,
  \end{equation}
and so there exists a constant $\varepsilon_1\in(0,1)$ such that for all $x$, $y\in\R^d$ with
$|x-y|\le \varepsilon_1$,
  $$\varphi'(2|x-y|)|x-y|^{2-\alpha(x)\wedge\alpha(y)}\le\varphi'(2|x-y|)|x-y|^{2-\alpha_0}\le -2C_2/C_1.$$
This, along with \eqref{f-cor-2}, \eqref{iii} and the fact that
$\varphi'(r)<0$ for $r\in(0,1]$, yields that there exist constants
$\varepsilon_2\in (0,\varepsilon_1)$ and $C_4>0$ such that for all
$n\ge1$, $\varepsilon\in(0,\varepsilon_2]$ and any $x$, $y\in \R^d$
with $1/n\le |x-y|< \varepsilon_2$,
  $$\widetilde{L} f_{n}(|x-y|)\le C_4\varphi'(2|x-y|) |x-y|^{2-\alpha_0}.$$
Again by $\varphi\in \mathscr{D}_{\alpha_0}$, we know that the function $r\mapsto \varphi'(2r)r^{2-\alpha_0}$
is increasing near the origin. Combining it with the inequality above, we prove the required assertion.
\end{proof}

Having Corollary \ref{f-cor-1} at hand, we can follow the argument of Theorem \ref{th1.1} to prove Theorem \ref{f-thm}. Here, we only present the

\begin{proof}[Sketch of the Proof of Theorem $\ref{f-thm}$]
We use the notation in the proof of Theorem \ref{th1.1}. By \eqref{f-cor-3}, for any
$\varepsilon\in(0,\varepsilon_0]$ and $x,y\in\R^d$ with $|x-y|\le \varepsilon$,
  $$\widetilde{\Ee}^{(x,y)}(T_{}\wedge S_\varepsilon)
  \le \frac{\int_0^{|x-y|}\varphi(s)\,ds}{-C_0\varphi'(2\varepsilon) \varepsilon^{2-\alpha_0}}$$ and  $$\widetilde{\Pp}^{(x,y)}(T>S_\varepsilon)\le\frac{\int_0^{|x-y|}\varphi(s)\,ds}{\int_0^\varepsilon\varphi(s)\,ds}.$$
This further gives us that for any $h\in B_b(\R^d)$ and $t>0$,
  $$\sup_{x\neq y}\frac{{|P_t h(x)-P_t h(y)|}}{\int_0^{|x-y|}\varphi(s)\,ds}\le
  2\|h\|_\infty \left[\frac{1}{\int_0^\varepsilon\varphi(s)\,ds}-\frac{1}{t C_0\varphi'(2\varepsilon) \varepsilon^{2-\alpha_0}}\right]
  , $$
  which in turn yields the desired assertion.
\end{proof}

\noindent\textbf{Acknowledgements.} We would like to thank
Professors Zhen-Qing Chen, Feng-Yu Wang and Xicheng Zhang for a
number of suggestions and helpful comments. The first author is
grateful to the financial supports of NSFC (No. 11101407), and the
Key Laboratory of RCSDS, CAS (No. 2008DP173182). The second author
would like to thank the research grants from the National Natural
Science Foundation of China (No.\ 11201073) and the Program for New
Century Excellent Talents in Universities of Fujian (No.\ JA12053).

\end{document}